\documentclass[reqno,11pt]{amsart}

\usepackage[a4paper,left=23mm,right=23mm,top=30mm,bottom=30mm,marginpar=25mm]{geometry}
\usepackage{amsmath}
\usepackage{amssymb}
\usepackage{amsthm}
\usepackage{amscd}
\usepackage[ansinew]{inputenc}
\usepackage{color}
\usepackage[final]{graphicx}
\usepackage{esint}

\renewcommand{\epsilon}{\varepsilon}

\numberwithin{equation}{section}

\newtheoremstyle{thmlemcorr}{10pt}{10pt}{\itshape}{}{\bfseries}{.}{10pt}{{\thmname{#1}\thmnumber{ #2}\thmnote{ (#3)}}}
\newtheoremstyle{thmlemcorr*}{10pt}{10pt}{\itshape}{}{\bfseries}{.}\newline{{\thmname{#1}\thmnumber{ #2}\thmnote{ (#3)}}}
\newtheoremstyle{defi}{10pt}{10pt}{\itshape}{}{\bfseries}{.}{10pt}{{\thmname{#1}\thmnumber{ #2}\thmnote{ (#3)}}}
\newtheoremstyle{remexample}{10pt}{10pt}{}{}{\bfseries}{.}{10pt}{{\thmname{#1}\thmnumber{ #2}\thmnote{ (#3)}}}
\newtheoremstyle{ass}{10pt}{10pt}{}{}{\bfseries}{.}{10pt}{{\thmname{#1}\thmnumber{ A#2}\thmnote{ (#3)}}}

\theoremstyle{thmlemcorr}
\newtheorem{theorem}{Theorem}
\numberwithin{theorem}{section}
\newtheorem{lemma}[theorem]{Lemma}
\newtheorem{corollary}[theorem]{Corollary}
\newtheorem{proposition}[theorem]{Proposition}

\theoremstyle{thmlemcorr*}
\newtheorem{theorem*}{Theorem}
\newtheorem{lemma*}[theorem]{Lemma}
\newtheorem{corollary*}[theorem]{Corollary}
\newtheorem{proposition*}[theorem]{Proposition}
\newtheorem{problem*}[theorem]{Problem}
\newtheorem{conjecture*}[theorem]{Conjecture}

\theoremstyle{defi}
\newtheorem{definition}[theorem]{Definition}

\theoremstyle{remexample}
\newtheorem{remark}[theorem]{Remark}
\newtheorem{example}[theorem]{Example}

\theoremstyle{ass}

\newcommand{\Acal}{\mathcal{A}}

\newcommand{\Ecal}{\mathcal{E}}

\newcommand{\Mcal}{\mathcal{M}}

\newcommand{\Pcal}{\mathcal{P}}

\newcommand{\Scal}{\mathcal{S}}

\newcommand{\Ibb}{\mathbb{I}}

\DeclareMathOperator{\rank}{rank}
\DeclareMathOperator{\trace}{tr}

\newcommand{\abs}[1]{|#1|}

\newcommand{\dd}{\;\mathrm{d}}

\newcommand{\N}{\mathbb{N}}
\newcommand{\R}{\mathbb{R}}

\newcommand{\eps}{\epsilon}

\newcommand{\cofoff}{\cof_{\rm off}\,}
\newcommand{\cofdiag}{\cof_{\rm diag}\, }

\newcommand{\cof}{{\rm cof\,}}

\newcommand{\EA}{\begin{equation}}
\newcommand{\EE}{\end{equation}}
\newcommand{\DMA}{\begin{displaymath}}
\newcommand{\DME}{\end{displaymath}}

\def\XXint#1#2#3{{\setbox0=\hbox{$#1{#2#3}{\int}$}
\vcenter{\hbox{$#2#3$}}\kern-.5\wd0}}

\title[Characterizations of symmetric polyconvexity]{Characterizations of symmetric polyconvexity}

\author{Omar Boussaid}
 \address{Faculty of Exact and Computer Sciences, Hassiba Ben Bouali University, Ouled Fares,  02180 Chlef, Algeria}
  \email{o.boussaid@univhb-chlef.dz}
\author{Carolin Kreisbeck}
\address{Mathematical Institute, Utrecht University, Postbus 80010, 3508 TA Utrecht, The Netherlands}
\email{c.kreisbeck@uu.nl}
\author{Anja Schl\"omerkemper}
 \address{Institute of Mathematics, University of W\"urzburg, Emil-Fischer-Str. 40, 97074 W\"urzburg, Germany}
 \email{anja.schloemerkemper@mathematik.uni-wuerzburg.de}
\begin{document}

%%%%%%%%%%%%%%%%%%%%%%%%%%%% ABSTRACT %%%%%%%%%%%%%%%%%%%%%%%%%%%%%%%%%%%

\maketitle

\thispagestyle{empty}
\begin{abstract}
Symmetric quasiconvexity plays a key role for energy minimization in geometrically linear elasticity theory. Due to the complexity of this notion, a common approach is to retreat to necessary and sufficient conditions that are easier to handle. 
This article focuses on symmetric polyconvexity, which is a sufficient condition. We prove a new characterization of symmetric polyconvex functions in the two- and three-dimensional setting, and use it to investigate relevant subclasses like symmetric polyaffine functions and symmetric polyconvex quadratic forms. In particular, we provide an example of a symmetric rank-one convex quadratic form in 3d that is not symmetric polyconvex. The construction takes the famous work by Serre from 1983 on the classical situation without symmetry as inspiration. Beyond their theoretical interest, these findings may turn out useful for computational relaxation and homogenization. 
\vspace{8pt}

\vspace{8pt}

\vspace{8pt}

\noindent\textsc{Date:} \today.
\end{abstract}

%%%%%%%%%%%%%%%%%%%%%%%%%%%% MAIN PART %%%%%%%%%%%%%%%%%%%%%%%%%%%

\section{Introduction}\label{sec:introduction} 

Variational models based on energy minimization principles are known to yield good descriptions of the elastic materials in nonlinear elasticity (so-called hyperelastic materials), and have inspired new mathematical developments in the calculus of variations over the last decades. \color{black}  Typically, one encounters elastic energies of the form
\begin{align}\label{elasticenergy1}
E(v)= \int_{\Omega} W(\nabla v)\dd{x},
\end{align}
where $\Omega\subset \R^d$, $d=2,3$, is a bounded domain representing the reference configuration of the elastic body, $v:\Omega\to \R^d$ is the deformation of the body, and $W:\R^{d\times d}\to \R$ an elastic energy density. Depending on the specific model,~\eqref{elasticenergy1} can be complemented with external force terms and/or the class of admissible deformations be restricted by boundary conditions. As is well-known in continuum mechanics, mechanically relevant energy densities cannot be convex due to incompatibility with the concept of frame-indifference, see e.g.~\cite{Cia88}. The weaker notion of polyconvexity, which was introduced by Ball in 1977~\cite{Bal77}, however, has turned out to be particularly suitable here, since $(i)$ existence of weak solutions can be shown (using the direct method in the calculus of variations) and $(ii)$ there exist realistic models (involving orientation preservation and avoiding infinite compression) that have polyconvex energies. Explicit polyconvex energies for certain isotropic and anisotropic materials were derived in e.g.~\cite{HarNef03,Mie05,Ste03}.
We recall that a function $\R^{d\times d}\to \R$ is called polyconvex if it depends in a convex way on its minors;
 in 3d, second and third order minors correspond to the matrix of cofactors and the determinant, respectively, while the second order minor in 2d is just the determinant.

In this article, we investigate functions that are polyconvex when composed with the linear projection of $\R^{d\times d}$ onto the subspace of symmetric matrices $\Scal^{d\times d}$, that is, $f:\Scal^{d\times d}\to \R$ such that $\R^{d\times d}\ni F\mapsto f(F^s)$ is polyconvex. We call such functions {\em symmetric polyconvex}, see Definition~\ref{def:sqc} for more details. 
This notion has applications in the geometrically linear theory of elasticity, which results from nonlinear elasticity theory by replacing the requirement of frame-indifference with the assumption that the elastic energy density is invariant under infinitesimal rotations, see e.g.,~\cite{Bha93, Bha03, Cia88}. In this case, the energy densities depend only on the symmetric part of the deformation gradient, or on linear strains $e(u)= \frac{1}{2}(\nabla u + (\nabla u)^T)$, where $u:\Omega\to \R^d$, $u(x) = v(x) - x$, is the elastic displacement field. That is, 
\begin{align}\label{elasticenergy2}
\Ecal(u)= \int_{\Omega} f(e(u))\dd{x},
\end{align}
with a density $f:\Scal^{d\times d}\to \R$. Here, we are particularly interested in the case when $f$ is symmetric polyconvex. The overall aim is to identify necessary and sufficient conditions for symmetric polyconvexity in 2d and 3d. 
As we will discuss further below, these conditions will facilitate finding explicit symmetric polyconvex functions that can serve as interesting energy densities in continuum mechanics or that provide lower bounds in the relaxation of models for materials with microstructures. Moreover, a deeper understanding of symmetric polyconvexity may be useful in further applications of the translation method.\\

Next, let us explain in more detail our new mathematical results, which were
inspired by the following examples of functions defined on $\Scal^{d\times d}$: 
\begin{itemize}
 \item[$(i)$] $\eps\mapsto \det \eps$ is not symmetric polyconvex in $d=2,3$;
 \item[$(ii)$] $\eps\mapsto - \det \eps$ is symmetric polyconvex in $d=2$, but not in $d=3$;
 \item[$(iii)$] $\eps\mapsto -(\cof\eps)_{ii}$, $i=1,2,3$ is symmetric polyconvex in $d=3$, while $\eps \mapsto \cof \eps$ is not.
\end{itemize}
Note that this is different from the classical setting, where $\R^{d\times d}\ni F\mapsto \pm \det F$ is polyconvex in any dimension, as are all minors. We discuss these functions in Example~\ref{ex:pmdet}, and give further examples in  Remarks~\ref{rem:nonuniqueness}\,b) and \ref{remrepre}.

We show that symmetric polyconvexity in 2d and 3d can be characterized as follows: Any symmetric polyconvex function in 2d corresponds exactly to a convex function of all minors that is non-increasing with respect to the determinant (see Theorem~\ref{prop:charpoly2d}). In particular, this includes $(ii)$ above and elucidates why $(i)$ is not symmetric polyconvex. 
For 3d, we prove in Theorem~\ref{representation} that any symmetric polyconvex function can be  represented as a convex function of first and second order minors (hence, no dependence on the determinant) whose subdifferential with respect to its cofactor variable is negative semi-definite. Notice that this result is in correspondence with example $(iii)$ above. 
The difficulty here lies in identifying these characterizations. The proofs build on monotonicity properties of convex functions expressed in terms of their (partial) subdifferentials (cf.~Lemmas~\ref{lem:aux} and~\ref{lem:aux2}), as well as on some standard tools from convex analysis, properties of semi-convex functions in the classical setting, and on basic algebraic relations for minors.  

Moreover, these findings can be used to investigate the two subclasses of symmetric polyconvex quadratic forms and symmetric polyaffine functions. Indeed, it is not hard to see that the latter, i.e.,~functions $f:\Scal^{d\times d}\to \R$ such that $f$ and $-f$ are symmetric polyconvex, are always affine as we outline in Subsection~\ref{subsec:polyaffine} and Proposition~\ref{cor:polyaffine}. Note that this stands in contrast with the classical case, where for instance the function $\R^{d\times d}\ni F\mapsto \det F$ is polyaffine, but not affine. Hence, in the symmetric setting, there are no non-trivial Null-Lagrangians.
Regarding the class of symmetric polyconvex quadratic forms we give explicit characterizations, both for $d=2$ (Proposition~\ref{prop:quadforms2d}) and $d=3$ (Proposition~\ref{theoremquad}). More precisely, if $f:\Scal^{d\times d}\to \R$ is symmetric polyconvex, then 
\begin{align}\label{char23b}
f(\eps) = h(\eps) - \alpha \det \eps\quad \text{and} \quad f(\eps) = h(\eps) - A: \cof \eps 
\end{align}
for 2d and 3d, respectively, where $h:\Scal^{d\times d}\to \R$ is convex, $\alpha>0$, and $A\in \Scal^{3\times 3}$ is positive semi-definite. The case $d=2$ can be viewed as a refinement of the result by Marcellini in~\cite{Mar84} from the classical to the symmetric framework. For a related statement in the three-dimensional classical setting we refer to \cite[p.~192]{Dac08}. 

While in two dimensions, symmetric rank-one convexity and symmetric polyconvexity are equivalent for quadratic forms as a consequence of the corresponding well-known result for classical quadratic forms (see e.g.~\cite[Theorem~5.25]{Dac08}), we prove that this is not true for $d=3$.  
For quadratic forms on $\R^{3\times 3}$, this has been known since early work by Terpstra~\cite{Ter39}, which Serre in~\cite{Ser83} underpinned with an explicit example. Precisely, he showed that there exists $\eta>0$ such that
\begin{align}\label{Serre}
F\mapsto (F_{11} - F_{23}-F_{32})^2 + (F_{12}-F_{31} + F_{13})^2 + (F_{21} -F_{13} - F_{31})^2 + F_{22}^2 + F_{33}^2 -\eta |F|^2
\end{align}
is rank-one convex, but not polyconvex,  cf.~also~\cite[p.\,194]{Dac08}. 
More recently, Harutyunyan and Milton \cite{HaM15} gave another example,
namely
\begin{align}\label{Milton}
F\mapsto F_{11}^2+F_{22}^2+F_{33}^2 + F_{12}^2 + F_{23}^2 + F_{31}^2  - 2(F_{11}F_{22}  +F_{22}F_{33} + F_{33}F_{11}).
\end{align}
There, the proof of non-polyconvexity exploits critically that $F_{13}, F_{21}$ and $F_{32}$ do not appear in the formula, see \cite[Proof of Theorem~1.5\ $ii)$]{HaM15}. How to adapt the examples in \eqref{Serre} and \eqref{Milton} to the symmetric setting is not immediately obvious due to the kind of dependence on the off-diagonal entries. 
As one of the main results of this article, we provide a symmetric variant of Serre's example: 
Theorem~\ref{theo:counterexample} states the existence of $\eta>0$ such that the quadratic form $f:\Scal^{3\times 3}\to \R$ given by 
\begin{align}\label{ourexample}
f(\eps)= (\eps_{13}-\eps_{23})^2+(\eps_{12}-\eps_{13})^2+(\eps_{12}-\eps_{23})^2 + \eps_{11}^2+\eps_{22}^2+\eps_{33}^2-\eta|\eps|^2\end{align}
is symmetric rank-one convex, but not symmetric polyconvex. The proof that this function is not symmetric polyconvex is based on the $3$d-characterization in~\eqref{char23b} 
and on a careful study of the minimizers of  $f + \eta |\eps|^2$ in the set of compatible matrices with unit norm (see~Lemma~\ref{lem:properties_eps0} for the details).\\

We have organized the paper as follows. In Section~\ref{sec:notions_convexity}, we recall the common generalized notions of convexity in the symmetric setting and review selected results from the literature regarding their characterization, properties and relations. The section is concluded with the above mentioned simple motivating examples. After some notational remarks and preliminaries in Section~\ref{sec:preliminaries}, we address characterizations of symmetric polyconvexity in $2$d in Section~\ref{sec:2d} and in $3$d in Section~\ref{sec:3d}. Moreover, we consider symmetric polyconvex quadratic forms in Subsections~\ref{sec:spc2} and \ref{sec:spc3}, where we also present the proof of Theorem~\ref{theo:counterexample}. Finally, symmetric polyaffine functions are the topic of Subsections~\ref{subsec:polyaffine} and \ref{subsec:polyaffine3}. In the remainder of the introduction, we comment further on the relevance of the notion of (symmetric) polyconvexity in the calculus of variations and elasticity theory.\\

Polyconvexity is a sufficient condition for quasiconvexity, which has been the subject of intensive investigation since its introduction by Morrey~\cite{Mor52} in 1952. It constitutes the central concept for the existence theory of vectorial integral problems in the calculus of variations, generalizing the notion of convexity to multi-dimensional variational problems.
If we consider the integral functional $E$ in \eqref{elasticenergy1} with suitable assumptions on the integrand $W$ and the space of deformations, quasiconvexity of $W$ is necessary and sufficient  for weak lower semicontinuity of $E$. Along with coercivity, quasiconvexity is thus a key ingredient for proving existence of minimizers of $E$ (e.g.~among all deformations with given boundary values) via the direct method in the calculus of variations, see e.g.\ \cite{Bal77,BeK17,Cia88,Dac08}. As for
weak lower semicontinuity of the functional $\Ecal$ in the geometrically linear theory of elasticity, cf.~\eqref{elasticenergy2}, symmetric quasiconvexity of the elastic energy density $f$ takes over the role of quasiconvexity of $W$~\cite{BFT00, Ebo00, FoM99}. This means that $f$ satisfies  Jensen's inequality for all symmetrized gradient test fields, or equivalently that the function is quasiconvex when composed with the linear projection of $\R^{d\times d}$ to symmetric matrices, cf.~Proposition~\ref{theo:characterization}. 
For a further discussion on the connections between linearized (small-strain) and nonlinear (finite-strain) elasticity theory
we refer the reader e.g.~to \cite{BaJ87, Bha93, DNP02, Koh90}.

Despite its importance, the notion of quasiconvexity is still not fully understood in all its facets due to its complexity. 
A common approach is therefore to retreat to the weaker and stronger conditions like rank-one convexity and polyconvexity, as they are easier to deal with and yet allow to make useful conclusions about quasiconvex functions~\cite{Dac08, Mul99}. The study of generalized notions of convexity for functions with specific properties has helped to gain valuable new insight and to advance the field. Relevant classes include one-homogeneous functions~\cite{DaH96, Mul92}, functions obtained by compositions with transposition~\cite{Kru99, Mul00}, or functions with different types of invariances, such as isotropic functions~\cite{Dac08,DaK93, MGN17, Mie05, Ros98, Sil99, Sil16}, quadratic forms with linearly elastic cubic, cyclic and axis-reflection symmetry~\cite{HaM15}, and functions of linear strains~\cite{ChB08,ChS13,Koh91,Pei13,Zha02,Zha03a}. 
In this spirit, the characterizations of symmetric polyconvexity proved in this article contribute to a deeper understanding of quasiconvex functions that are invariant under symmetrization. 

Even though minimization problems with integral functionals whose density is not (symmetric) quasiconvex may in general not admit solutions, they are highly relevant in many applications as they allow to model oscillation effects, like microstructure formation in materials, see e.g.~\cite{Mul99}.  
The mathematical task is then to describe the effective behavior of the model by analyzing minimizing sequences or low energy states. This comes down to relaxation, i.e.~to finding the lower semi-continuous envelope of the given functional. 
In analogy to the relaxation of integral functionals in the nonlinear setting (see e.g.~\cite{Dac08,Mul99}), this is achieved in the symmetric setting by the symmetric quasiconvexification of the integrand.
The symmetric quasiconvex envelope of a function $f$, i.e.~the largest symmetric quasiconvex function not larger than $f$, can be represented as 
\begin{align}\label{fsqc}
f^{\rm sqc}(\eps)=\inf_{\varphi\in C^\infty_c((0,1)^d;\R^d)}\int_{(0,1)^d} f(\eps + e(\varphi))\dd{x}, \qquad \eps\in \Scal^{d\times d},\end{align}
see e.g.~\cite[Proposition~2.1]{Zha04}.
We point out that explicit calculations of $f^{\rm sqc}$ tend to be rather challenging, since they require to solve the infinite-dimensional minimization problem in~\eqref{fsqc} for every symmetric matrix $\eps$.
A common strategy is to approach this problem 
by searching for suitable upper and lower bounds, typically in the form of (symmetric) polyconvex and (symmetric) rank-one convex envelopes, cf.\ e.g.\ \cite{ChS13, CoD16, CDK13, CoT05, FHS16, GMH03, LDR95}. In case of matching bounds and non-extended valued densities, one obtains even an exact relaxation formula.
Aside from applications in linearized elasticity, the characterizations of symmetric polyconvexity provided in this paper are potentially useful for relaxation problems arising from elasto-plastic models or in the theory of liquid crystals.

Another technique from homogenization and optimization theory, called the translation method, has also proven itself very successful in deriving good lower estimates on quasiconvex envelopes. For an introduction and a historical overview we refer to \cite[Chapter~8]{Che12}, for publications related to elasticity theory see~e.g.~\cite{CheBha08,Fir91,Koh91,Pei13} and the references therein. 
In our symmetric setting, the translation method can be briefly summarized as follows. We observe that $f = f- q + q \geq (f-q)^c + q$ for any $f,q:\Scal^{d\times d}\to \R$, where the superscript ${\rm c}$ stands for the convex envelope. If the so-called translator $q$ is symmetric quasiconvex, $(f-q)^c + q$ is symmetric quasiconvex. Hence, 
\begin{align}\label{translation}
f^{\rm sqc}\geq (f-q)^{\rm c} +q.
\end{align}
Optimizing the right-hand side of~\eqref{translation} over a subclass of symmetric quasiconvex functions provides a lower bound on $f^{\rm sqc}$. We remark that if one took the supremum over the totality of all symmetric quasiconvex functions $q$, for which explicit representations are not available, though, the method would be exact. Hence, a good choice of the class of translators in the sense of finding the balance between generality and explicitness determines the effectiveness of the method. There are two natural options to be used here: $(i)$ specific symmetric polyconvex functions, which, thanks to our characterizations in Theorems~\ref{prop:charpoly2d} and~\ref{theo:characterization3d}, can be easily constructed, and $(ii)$ symmetric rank-one convex quadratic forms, cf.\ also ~\cite{Zha03a}.

Our characterization of symmetric polyconvex  (and thus rank-one convex) quadratic forms in 2d (cf.~\eqref{char23b} or Proposition~\ref{prop:quadforms2d}) yields an explanation of why the translator $\eps\mapsto -\det\eps$ is often a good choice. It was used for example in the 2d setting of \cite{CheBha08} in the derivation of a relaxation formula for two-well energies with possibly unequal moduli. 
Indeed, if we rewrite the right hand side of \eqref{translation} as $(f-q)^{\rm c} - h^{\rm c} + h + q \geq (f- h - q)^{\rm c}  + h + q$ with a convex quadratic function $h:\Scal^{2\times 2}\to \R$, we see in view of the first equation in~\eqref{char23b} that working with just $\eps\mapsto -\alpha\det \eps$ for $\alpha>0$ as a translator is equivalent to using all symmetric rank-one convex quadratic forms. 
The second equation in~\eqref{char23b} indicates that the analogous observation is true in three dimensions for $\eps\mapsto -A:\cof \eps$ with $A\in \Scal^{3\times 3}$ positive semi-definite. Translators of this type play a key role in the derivation of the bounds in~\cite{CheBha08, Pei13}. Our characterization result hence provides structural insight into the choice of translators in the above-mentioned literature. 

In the classical setting, Firoozye \cite{Fir91} showed that a translation bound optimized over rank-one convex quadratic forms and Null-Lagrangians is at least equally good as polyconvexification, and even strictly better for some three-dimensional functions. His proof of this latter statement is based on Serre's example in \eqref{Serre}. 
Our example~\eqref{ourexample} in the 3d symmetric setting clearly implies that, in contrast to 2d, considering symmetric rank-one convex quadratic forms as translators will in general give better bounds than using just symmetric polyconvex ones. \color{black} Whether there are situations when combining symmetric rank-one convex quadratic forms with other symmetric polyconvex functions leads to improved results remains an open question for future research; notice that we do not have any non-trivial Null-Lagrangians at hand in the symmetric setting, cf.\ Proposition~\ref{cor:polyaffine}. 

%
%%%%%%%%%%%%%%%%%%%%%%%%%%%%%%%%%%%%%%%%%%%%%%%%%%%%%%%%%%%%%%%%%%
\section{Different notions of symmetric semi-convexity}\label{sec:notions_convexity}

When speaking of semi-convexity, we will always refer to one of the following notions of generalized convexity: quasiconvexity, polyconvexity and rank-one convexity.
Let us briefly recall that a function $f:\R^{d\times d}\to \R$ is called quasiconvex if it satisfies Jensen's inequality for all gradient test fields, precisely,
\begin{align*}
f(F)\leq \inf_{\varphi \in C^\infty_c((0,1)^d;\R^d)}\int_{(0,1)^d} f(F+\nabla \varphi)\dd{x} \qquad \text{for all $F\in \R^{d\times d}$,}
\end{align*}
assuming the integrals on the right-hand side exists. 
A function $f:\R^{d\times d}\to \R$ is
polyconvex if it can be written in terms of a convex function of its minors. Finally, $f:\R^{d\times d}\to \R$ is a rank-one convex function, if it is convex along all rank-one lines $t\mapsto F+t a\otimes b$ for $F\in \R^{d\times d}$ and $a, b\in \R^d$, where $(a\otimes b)_{ij} = a_ib_j$ for $i,j=1, \ldots, d$. 
For more details, see e.g.~the standard work by Dacorogna~\cite{Dac08}.

Here, we are interested in functions that are independent of the skew-symmetric part of its variables $F\in \R^{d\times d}$, that is, functions that depend only on the symmetric part of $F$ as motivated by geometrically linear elasticity theory, cf.~Section~\ref{sec:introduction}. As documented there, also here the semi-convexity notions are of interest, i.e.,  symmetric quasi-, poly- and rank-one convex functions. 
The special class of semi-convex functions that are independent of skew-symmetric parts motivates the concept of  symmetric semi-convexity.  According to the following definition, we call a function defined on the space of $\Scal^{d\times d}$ symmetric semi-convex, if its natural extension to all matrices in $\R^{d\times d}$ is semi-convex in the conventional sense, cf.~work by Zhang~\cite{Zha02, Zha04}, where symmetric semi-convexity is called semi-convexity on linear strains.
\begin{definition}[Symmetric semi-convex functions]\label{def:sqc}
A function $f:\Scal^{d\times d}\to \R$ is symmetric semi-convex, if the function 
\begin{align*}
\tilde{f}= f\circ \pi_d:\R^{d\times d}\to \R
\end{align*} 
is semi-convex. Here, $\pi_d( F) = \tfrac12 (F + F^T) = F^s$, $F \in \R^{d\times d}$, is the orthogonal projection onto the subspace of symmetric matrices.  
\end{definition}
Note that in particular, $\tilde f(F) = f(F^s) = \tilde f(F^s) = \tilde f(F^T)$ for any $F\in \R^{d\times d}$, i.e., $\tilde f$ is invariant under symmetrization. 
As an aside, we mention that a corresponding definition of symmetric convex functions is possible. By linearity of $\pi_d$, a function is symmetric convex if and only if it is convex.

Next we will collect and review some classical, as well as more recent, results in the context of symmetric semi-convex functions. The following characterizations of symmetric quasi- and rank-one convexity for general dimensions $d$ are straightforward to show and appear to be well-known, see e.g.~\cite{Ebo00, Zha02}. 
\begin{proposition}\label{theo:characterization}
Let $f:\Scal^{d\times d}\to\R$. Then 
\begin{itemize}
\item[$i)$] $f$ is symmetric quasiconvex if and only if for every $\eps\in \Scal^{d\times d}$,
\begin{align}\label{theo_char1}
f(\eps)\leq \inf_{\varphi\in C^{\infty}_c((0,1)^d;\R^d)}\int_{(0,1)^d} f(\eps+(\nabla \varphi)^s)\dd{x}. 
\end{align} 
\item[$ii)$] $f$ is symmetric rank-one convex if and only if
\begin{align}\label{theo_char2}
f(\lambda \eps+(1-\lambda)\eta)\leq \lambda f(\eps)+(1-\lambda)f(\eta)
\end{align}
for all $\lambda\in (0,1)$ and $\eps, \eta\in \Scal^{d\times d}$ compatible, i.e.~$\eps-\eta = a\odot b := \tfrac{1}{2}(a\otimes b + b\otimes a)$ for some $a, b \in \R^{d}$. \\
Equivalently, $t\mapsto f(\eps + t a\odot b)$ is convex for any $\eps \in \Scal^{d\times d}$ and any $a,b \in \R^d$. 
\end{itemize}
\end{proposition}

\begin{remark}
a) Notice that many works involving semi-convex functions defined on linear strains, such as~\cite{ChS13, PhR17, Ebo00,  Rin11}, take the characterizations of Proposition~\ref{theo:characterization} as a starting point and definition.

b) In~\cite{Sve92a}, a function $f:\Scal^{d\times d}\to \R$ is called quasiconvex, if for every $\eps\in \Scal^{d\times d}$,
\begin{align*}
f(\eps)\leq \inf_{\psi\in C_c^{\infty}((0,1)^d)} \int_{(0,1)^d} f(\eps + D^2\psi)\dd{x} 
\end{align*}
where $D^2\psi$ denotes the Hessian matrix of $\psi$, cf.~also~\cite{FaZ03}. 
This notion is strictly weaker than the symmetric quasiconvexity in the sense of Definition~\ref{def:sqc}. 
Since for every $\psi\in C_c^{\infty}((0,1)^d)$ 
the gradient $\nabla \psi$ is an admissible test field in~\eqref{theo_char1}, the asserted implication is immediate. To see that it is strict, we consider in $2$d, the function 
\begin{align*}
f_0(\eps) = \begin{cases}
\det \eps & \text{if $\eps$ is positive definite},\\
0 & \text{otherwise,}
\end{cases} \qquad \eps\in \Scal^{2\times 2},
\end{align*}
which \v{S}ver\'{a}k in~\cite{Sve92a} proved to be quasiconvex. However, $f_0$ is not symmetric rank-one convex, and therefore not symmetric quasiconvex (see~\eqref{implications} below), since the following map, which is the composition of a compatible line with $f_0$, is not convex: 
\begin{align*}
\R\to \R, \quad t\mapsto f_0(\Ibb+ 2te_1\odot e_2) =
\begin{cases} 
\det(\Ibb+ 2te_1\odot e_2) = 1-t^2 & \text{for $t\in (-1,1)$}\\ 
0 &\text{otherwise,} 
\end{cases}
\end{align*}
where $\Ibb$ is the identity matrix, and $e_1, e_2$ are the standard unit vectors in $\R^2$.

Similarly, we show in the $3$d setting that the following quasiconvex functions from~\cite{Sve92a}
\begin{align*}
f_l(\eps)  = \begin{cases}
\abs{\det \eps} & \text{if $\eps$ has exactly $l$ negative eigenvalues},\\ 
0 & \text{otherwise,}
\end{cases} \qquad \eps\in \Scal^{3\times 3},
\end{align*}
with $l=1,2$, are not symmetric rank-one convex. For $l=1$ and $l=2$, we use the compatible lines $t\mapsto {\rm diag}(1,1,-1) + 2te_1\odot e_2$ and $t\mapsto {\rm diag}(-1,-1, 1) + 2 t e_1\odot e_2$,  respectively.   

c) Linearized elasticity can be viewed within the general $\Acal$-free framework~\cite{FoM99, Mur81, Tar79}. With the second-order constant-rank operator $\Acal$ defined for $V\in C^\infty((0,1)^d;\R^{d\times d})$ by
\begin{align*}
(\Acal V)_{jk} = \sum_{i=1}^d \partial_{ik}^2 V_{ji} + \partial_{ij}^2 V_{ki}  - \partial^2_{jk}V_{ii} - \partial_{ii}^2 V_{jk}, \quad  j,k=1, \ldots, d,
\end{align*}
one has that $\Acal V=0$ in $(0,1)^d$ if and only if $V=(\nabla u)^s$ for some $u\in C^\infty((0,1)^d; \R^d)$. Consequently,~\eqref{theo_char1} corresponds to $\Acal$-quasiconvexity~\cite{FoM99}, while~\eqref{theo_char2} is equivalent to convexity along directions in the characteristic cone
\begin{align*}
\Lambda_\Acal = \cup_{\xi\in \R^d}\ker  \mathbb{A}(\xi) = \{a\odot b: a,b\in \R^d\}.
\end{align*}
Here $\mathbb{A}(\xi)V = V\xi\otimes \xi + \xi \otimes V\xi - (\trace V)\xi\otimes \xi -\abs{\xi}^2 V$ for $\xi\in \R^d$, $V\in \R^{d\times d}$ is the symbol of $\Acal$, and $\ker \mathbb{A}(\xi) = \{a\odot \xi: a\in \R^d\}$. For more details, see ~\cite[Example~3.10]{FoM99}.

d) Contrary to expectations that may arise in the light of Proposition~\ref{theo:characterization}, symmetric polyconvexity of a function $f:\Scal^{d\times d}\to \R$ according to Definition~\ref{def:sqc} is not the same as $f$ being a convex function of symmetric quasiaffine maps (or Null-Lagrangians). Indeed, since there are no non-trivial Null-Lagrangians in the symmetrized context (cf.~Section~\ref{subsec:polyaffine} for $d=2$ and Proposition~\ref{cor:polyaffine} for $d=3$), the latter property equals convexity of $f$, and is strictly stronger than symmetric polyconvexity, cf.~Example~\ref{ex:pmdet}.% 
\end{remark}

The relation between the different notions of symmetric semi-convexity is an immediate consequence of the  implications for their classical versions without symmetry, see~\cite{Dac08}. Hence, it holds for $f:\Scal^{d\times d}\to \R$ that
\begin{align}
\begin{array}{l}\label{implications}
f\ \text{\rm is convex} \Rightarrow f\ \text{is symmetric polyconvex} \\ 
\hspace{2cm}\Rightarrow f\ \text{is symmetric quasiconvex} \Rightarrow f\ \text{is symmetric rank-one convex.}
\end{array}
\end{align}

Equivalence in~\eqref{implications} is in general not true. Counterexamples that are commonly cited in the classical setting for finite-valued functions on $\R^{d\times d}$ are for instance $F\mapsto \det F$, the parameter-dependent example of~\cite{AlD92, DaM88} for $d=2$, the famous counterexample by \v{S}ver\'ak~\cite{Sve92b}, which shows that quasiconvexity is strictly stronger than rank-one convexity if $d\geq 3$, and the example of $3$d rank-one convex, but not polyconvex quadratic forms in~\cite{HaM15,Ser83}. All these counterexamples   depend in a non-trivial way on the skew-symmetric parts, and are hence not suitable in the context of functions on symmetric matrices. 
In ~\cite{Ebo00,Koh91}, one finds examples of symmetric quasiconvex functions that are not convex, which have resulted from relaxation of double-well functions.
We will come back to the discussion of why the reverse implications do not hold, see in particular Example~\ref{ex:pmdet} and the proof of Theorem~\ref{theo:counterexample}.  \\

An important class of semi-convex functions with a very long history of intensive study are quadratic ones, for more recent work we refer e.g.~to~\cite{HaM15, HaM17,Zha03a}. It was shown by van Hove in \cite{Hov47, Hov49} that for general quadratic forms, the notions of quasiconvexity and rank-one convexity coincide. For $d=2$, the latter are even equivalent to polyconvexity, see~\cite[Theorem~5.25]{Dac08} and the references therein. In view of Definition~\ref{def:sqc}, these results carry over immediately to the symmetric setting, where
a quadratic form is any expression 
$$
f(\eps)=M\eps:\eps, \quad \eps\in \Scal^{d\times d},
$$
where $M$ is a fourth order tensor with the symmetries $M_{ijkl} = M_{lkij}=M_{ijlk}$. Note that a quadratic function $f$ on symmetric matrices is convex if and only if $f(\eps) \geq 0$ for any $\eps \in \Scal^{n\times n}$. 
A useful characterization for a quadratic form $f$ to be symmetric rank-one convex is that
\begin{align}\label{rankform}
f(a\odot b)\geq 0\qquad \text{ for all $a, b\in \R^d$,}
\end{align} 
as follows directly from Proposition~\ref{theo:characterization}\,$ii)$. 
Based on this characterization, Zhang~\cite{Zha03a}   classified the symmetric rank-one convex quadratic forms for $d=2,3$ via their nullity and Morse index. 
In parallel to the classical setting, symmetric rank-one convexity is strictly weaker than symmetric polyconvexity for $d=3$, as our example in Theorem~\ref{theo:counterexample} shows.

The following basic example served us as a motivation for the characterization results of symmetric polyconvex functions in Theorem~\ref{prop:charpoly2d} and Theorem~\ref{theo:characterization3d}. 

\begin{example}\label{ex:pmdet}
Let $d=2,3$. The determinant map $\Scal^{d\times d}\to \R$, $\eps\mapsto \det \eps$ is not symmetric rank-one convex, and therefore neither symmetric quasi- nor polyconvex.

In $2$d, $\eps\mapsto \det \eps$ is a quadratic form and one finds that for $a, b\in\R^2$, 
\begin{align*}
\det(a\odot b)= -\frac{1}{4}(a_1b_2-a_2b_1)^2 \leq 0. 
\end{align*} 
Hence, in view of~\eqref{rankform}, the determinant map is indeed not symmetric rank-one convex, but we find that $\eps\mapsto -\det \eps$ is. 
Since symmetric rank-one convexity is equivalent to symmetric polyconvexity for quadratic forms on $\Scal^{2\times 2}$, $\eps\mapsto -\det \eps$ is symmetric polyconvex,  and thus symmetric quasiconvex. 
A direct argument for the symmetric quasiconvexity of $\eps\mapsto -\det \eps$ uses~\eqref{det_2d}  below, the fact that for any antisymmetric matrix   $\det F^a \geq 0$ in 2d,  and exploits the Null-Lagrangian property of the determinant to conclude that 
$$
-\int_{(0,1)^2}\det(\epsilon+(\nabla \varphi)^s)\dd{x}=-\int_{(0,1)^2}\det (\epsilon+\nabla\varphi)+\int_{(0,1)^2}\det (\nabla\varphi)^a\dd{x} \geq   -\det \eps
$$
for all $\varphi\in C^{\infty}_c((0,1)^2;\R^2)$. Interestingly, the previous calculation remains valid if, instead of the function $-\det \eps$, we consider $f(\epsilon)=g(\epsilon,\det \epsilon)$ with $g:\R^5\to\R$ being a convex function that is non-increasing with respect to the last variable. Hence, every such function $f$ is symmetric quasiconvex. In Theorem~\ref{prop:charpoly2d}, we show that these functions in fact serve as characterization of symmetric polyconvexity in 2d. 

In the $3$d case, both $\eps\mapsto \det\eps$ and $\eps\mapsto - \det \eps$ fail to be symmetric rank-one convex. 
However, by taking the diagonal $2\times 2$ minors, simple examples of symmetric rank-one convex functions can be constructed. A direct adaptation of the $2$d argument above shows that the maps $\eps\mapsto-(\cof \epsilon)_{ii}$ for $\eps\in \Scal^{3\times 3}$ with $i=1,2,3$ are symmetric polyconvex, while $\eps\mapsto (\cof \epsilon)_{ii}$ are not. 
\end{example}

\section{Preliminaries}\label{sec:preliminaries}

Before proving the results announced in the introduction, we use this section to collect further relevant notation and auxiliary results.

\subsection{Notation.}\label{subsec:notations}
This work focuses on the space dimensions $d=2,3$. We write $a\cdot b$ with $a, b\in \R^d$ for the standard inner product on $\R^d$, and use the scalar product $A:B=\sum_{i, j=1}^d A_{ij} B_{ij}$ for $d\times d$ matrices $A$ and $B$. The latter induces the Frobenius norm $|A|^2:=A:A$ on $\R^{d\times d}$. Moreover,  $e_i$ with $i=1, \ldots, d$ are the standard unit vectors in $\R^d$,  $(a\otimes b)_{ij} = a_ib_j$ with $i, j=1, \ldots, d$ for $a, b\in \R^d$ is the tensor product of $a$ and $b$, and $a\odot b=\frac{1}{2}(a\otimes b + b\otimes a)$ with $a, b\in \R^d$.
Further, ${\rm diag}(\lambda_1, \ldots, \lambda_d)$ with $\lambda_i\in \R$ is our notation for diagonal $d\times d$ matrices. 

Let us denote by $\Scal^{d\times d}$ the set of symmetric matrices in $\R^{d\times d}$. Any $F\in \R^{d\times d}$ can be decomposed into its symmetric and antisymmetric part, i.e.~
$F=F^s+F^a$ with $F^s=\frac{1}{2}(F+F^T)\in \Scal^{d\times d}$ and $F^a=\frac{1}{2}(F-F^T)$.   For the subsets of positive and negative semi-definite matrices in $\Scal^{d\times d}$ we use the notations $\Scal^{d\times d}_+$ and $\Scal^{d\times d}_-$, respectively.  

In the $3$d case, the $2\times 2$ minors of $F\in \R^{3\times 3}$ are $M_{ij}(F) = \det \widehat{F}_{ij}$ for $i, j\in \{1,2,3\}$, where $\widehat{F}_{ij}$ stands for the matrix obtained from $F$ by deleting the $i$th row and the $j$th column. 
The cofactor matrix of $F$ is then given by
\begin{align*}
(\cof F)_{ij} = (-1)^{i+j} M_{ij}(F), \quad i,j\in \{1,2,3\}.
\end{align*} 
  One can split $\cof F\in \R^{3\times 3}$ into its diagonal and non-diagonal entries, denoted by $\cof_{\rm diag} F \in \R^3$ and $\cof_{\rm off} F\in \R^6$, respectively. Precisely, 
\begin{align*}
\cof_{\rm diag} F & = ((\cof F)_{11}, (\cof F)_{22}, (\cof F)_{33})^T 
\end{align*}
and
\begin{align*}
\cof_{\rm off} F & = ((\cof F)_{12}, (\cof F)_{13}, (\cof F)_{23}, (\cof F)_{21}, (\cof F)_{31}, (\cof F)_{32})^T.
\end{align*}
For symmetric $\eps\in \Scal^{3\times 3}$ we know that $\cof \eps$ is symmetric, so that we may take $\cof_{\rm off} \eps\in \R^3$, i.e., 
\begin{align*}%\label{symcof_3}
\cof_{\rm off} \eps& = ((\cof \eps)_{12}, (\cof \eps)_{13}, (\cof \eps)_{23})^T.
\end{align*}

\subsection{Properties of symmetric matrices and their minors.}
For $d=2$, we observe that
\begin{align}\label{det_2d}
\det F=\det F^s+\det F^a, \qquad F\in \R^{2\times 2},
\end{align}
and for $d=3$, straightforward calculations yield that 
\begin{align}\label{cof_3d}
\cof F^s = (\cof F)^s - \cof F^a, \qquad F\in \R^{3\times 3},
\end{align}
and
\begin{align*}
\det F=\det F^s + \cof F^a:F = \det F^s+\cof F^a:F^s, \qquad F\in \R^{3\times 3}.
\end{align*} 
Notice that $\cof F^a, \cof F^s\in \Scal^{3\times 3}$ for $F \in \R^{3\times 3}$. In general, $\cof F^T =(\cof F)^T$ for all $F\in \R^{3\times 3}$. 

Another useful formula is Cramer's rule, which states that for all $F\in \R^{3\times 3}$,
\begin{align}\label{Cramer}
(\det F)\Ibb = (\cof F)F^T.
\end{align}
Here $\Ibb$ is the identity matrix in $\R^{3\times 3}$.

If $F$ is antisymmetric, i.e.~$F=F^a$, there exist $x\in \R^3$ such that 
\begin{align*}
F=\begin{pmatrix} 0 & x_3 & -x_2\\ -x_3 & 0 & x_1 \\ x_2& -x_1 & 0 \end{pmatrix},
\end{align*}
and consequently,
\begin{align}\label{cofFa}
\cof F = \begin{pmatrix} x_1^2 & x_1x_2 & x_1x_3 \\ x_1x_2 & x_2^2 & x_2x_3\\  x_1x_3 & x_2x_3 & x_3^2 \end{pmatrix} = \begin{pmatrix} x_1 \\ x_2\\ x_3\end{pmatrix} \otimes \begin{pmatrix} x_1 \\ x_2 \\ x_3\end{pmatrix} \in \Scal^{3\times 3}. 
\end{align} 

Hence, we obtain for $a, b\in \R^3$ that
\begin{align}\label{cofab_anti}
\cof \bigl((a\otimes b)^a\bigr) =  \tfrac{1}{4} (a\times b)\otimes (a\times b), 
\end{align}
where $a\times b$ stands for the cross product of the vectors $a$ and $b$. In component notation, $(a\times b)_i = a_{j}b_{k}-a_{k}b_{j}$ for every circular permutation $(ijk)$ of $(123)$.
In view of~\eqref{cof_3d} and the elementary calculation that $\cof(a\otimes b)=0$, we also find that 
\begin{align}\label{cof_ab}
\cof(a\odot b) = -\tfrac{1}{4} (a\times b)\otimes (a\times b), \quad a, b \in \R^3.
\end{align}
 
These elementary observations allow us to prove the following useful lemma. 
\begin{lemma}\label{convexquadra}
For $A \in \Scal^{3\times 3}$ let $q_A:\R^{3\times 3}\to\R$ be the quadratic form given by
\begin{align}\label{QA}
q_A(F)= A:\cof F^a, \quad F\in \R^{3\times 3}.
\end{align}
Then the following three conditions are equivalent:
\begin{itemize}
\item[$i)$] $q_A$ is convex;
\item[$ii)$] $q_A$ is rank-one convex;
\item[$iii)$] $A$ is positive semi-definite. 
\end{itemize}
\end{lemma}
\begin{proof} Clearly, $i)$ implies $ii)$. 
For the implication $``ii) \Rightarrow iii)"$ take any $x\in \R^3$ and select vectors $a, b\in \R^3$ such that $a\times b=x$.
Then, by~\eqref{cofab_anti},
\begin{align*}
q_A(t a\otimes b) & = t^2 A:\cof((a\otimes b)^a) = \tfrac{1}{4}t^2 A: (x\otimes x) = \tfrac{1}{4} t^2 (Ax\cdot x)
\end{align*}
for all $t\in \R$. Since $q_A$ is convex along the rank-one line $t\mapsto t(a\otimes b)$, it follows that $Ax\cdot x\geq 0$. Hence, $A\in \Scal^{3\times 3}_+$.  

To prove $``iii) \Rightarrow i)"$ we show that $q_A(F)\geq 0$ for all $F\in \R^{3\times 3}$. By choosing $x\in \R^3$ 
such that 
\begin{align*}
F^a = \begin{pmatrix} 0 & x_3 & -x_2 \\  -x_3 & 0 & x_1 \\ x_2 & -x_1 & 0 \end{pmatrix},
\end{align*}
one obtains due to~\eqref{cofFa} and the positive definiteness of $A$ that
\begin{align*}
q_A(F) = A : \cof F^a = A: (x\otimes x) = A x\cdot x\geq 0.
\end{align*}
\end{proof}

We denote the convex hull of any set $S\subset \R^{d\times d}$ by $S^c$. Later on, we will utilize the characterizations of two specific convex hulls, namely
\begin{align}\label{convexhull_2d}
\{\eps\in \Scal^{2\times 2}: \det \eps=0\}^c=\Scal^{2\times 2}
\end{align} 
and 
\begin{align}\label{convexhull_3d} 
\{\eps\in \Scal^{3\times 3}: \cof \eps =0\}^c = \Scal^{3\times 3}. 
\end{align} 
This follows from the observation that the convex hull of a set $S\subset\R^{d\times d}$ with the property that $S=\alpha S$ for all $\alpha \in \R$ coincides with its linear span. For $S=\{\eps\in \Scal^{d\times d}: \det \eps=0\}$ the latter coincides with $\Scal^{d\times d}$.

\subsection{Convex functions, subdifferentials and monotonicity.}
For vectors $y, \bar{y}\in \R^n$ the order relation $y \leq \bar{y}$ is to be understood componentwise, that is as $y_i \leq \bar{y}_i$ for $i=1, \ldots, n$, and analogously for $y\geq \bar{y}$. Monotonicity of a function $h:\R^n\to \R$ is also defined componentwise, that is, $h$ is called non-increasing (non-decreasing) if $h(y)\geq h(\bar{y})$ for all $y, \bar{y}\in \R^n$ with $y \leq  \bar{y}$ ($y\geq \bar{y}$).  

For a convex function $h:\R^n\to \R$ and $y\in \R^n$, the non-empty, compact and convex set 
\begin{align}\label{subdifferential}
\partial h(y) = \{b\in \R^n: h(\tilde{y}) \geq h(y)   +  b\cdot (\tilde{y}-y)   \text{ for all $y, \tilde{y}\in \R^n$}\}
\end{align}
is the subdifferential of $h$ in $y$. Next, we summarize some further well-known facts for subdifferentials. An element $b\in \partial h(y)$ is commonly referred to as subgradient of $h$ in $y$, see e.g.~\cite{Dac08,Roc97}. The subdifferential was introduced to generalize the classical notion of differentiability to general convex functions. If $g$ is differentiable in $y$, then $\partial h(y)$ contains exactly one element, namely the gradient $\nabla h(y)$. 
In generalization of the smooth case, the function $h$ is non-increasing (non-decreasing) if and only if $\partial h(y)\subset (-\infty, 0]^n$  ($\partial h(y)\subset [0, \infty)^n$) for all $y\in\R^n$. The multivalued map $\partial h: \R^n\to \Pcal(\R^n)$, where $\Pcal(\R^n)$ is the power set of $\R^n$, is a monotone operator, meaning that 
\begin{align}\label{monotone_operator}
(b-\tilde{b})\cdot (y-\tilde{y})\geq 0 \qquad \text{for all $y, \tilde{y}\in \R^n$, $b\in \partial h(y)$, $\tilde b\in \partial h(\tilde y)$,}
\end{align}
cf.~\cite{Roc70}. Moreover, $\partial h$ is continuous in the sense that if $(y^{(j)})_j\subset\R^n$ is a sequence that converges to some $y\in \R^n$ and $\delta>0$, there exists an index $J\in \N$ such that 
\begin{align}\label{continuity}
\partial h(y^{(j)})\subset \partial h(y) + \delta B_1(0) \qquad \text{for every $j\geq J$,}
\end{align}
where $B_1(0)$ is the unit ball in $\R^n$ with center in the origin, see~\cite[Theorem~24.5]{Roc97}.

Next, we will collect and prove a few auxiliary results on (partial) subdifferentials and monotonicity properties of convex functions defined on the product space $\R^m\times \R^n$.
If $g: \R^m\times \R^n \to \R$ is convex, one can define the partial subdifferentials $\partial_i g(x, y)$ for $i=1,2$, via \eqref{subdifferential}  by freezing the dependence of one of the (vector) variables. 
Regarding the relation between the full and the partial subdifferentials, it is immediate to see that 
\begin{align}\label{inclusion}
\partial g(x, y) \subset \partial_1 g(x, y)\times \partial_2 g(x, y) \qquad \text{for every $(x, y)\in \R^{m\times n}$.}
\end{align} 

For the reverse inclusion, which is more subtle, we state the following version of~\cite[Theorem~3.3]{BeC99} adapted to the present situation. Let us remark that in the case of convex functions, the notion of Dini subgradient or Fr\'{e}chet subderivate used in~\cite{BeC99} coincides with the classical subgradients of convex analysis. A summary of results on Fr\'{e}chet subdifferentials can be found in the survey article~\cite{Kru03}.

\begin{lemma}\label{lem:partial_subdifferential}
Let $g:\R^m\times \R^n\to \R$ be convex. Then there exists a dense Borel set $S\subset \R^m\times \R^n$ whose complement is a Lebesgue null-set such that
 \begin{align*}
\partial g(x, y) =\partial_1 g(x, y) \times \partial_2 g(x, y) \quad \text{for every $(x, y)\in S$.}
\end{align*}
Moreover, for any $x\in \R^m$ the set $S_x=\{y\in \R^n: (x, y)\in S\}$ lies dense in $\R^n$ and its complement has zero Lebesgue measure.
\end{lemma}

The next lemma generalizes the elementary observation that every bounded and convex function $\R^n\to \R$ is constant to the situation where one has only partial bounds on the growth behavior of the function in selected variables. 

\begin{lemma}\label{lem:aux}
Let $g:\R^m\times \R^n \to \R$ be convex.
\begin{itemize}
\item[$i)$] The function $g$ is non-increasing (non-decreasing) in the second variable if and only if $\partial_2 g(x, y)\subset (-\infty, 0]^n$ ($\partial_2 g(x,y)\subset [0, \infty)^n$) for every $(x,y) \in \R^m\times \R^n$. 
\item[$ii)$] If $\partial g(x, y)\subset \R^m\times (-\infty, 0]^n$ ($\partial g(x, y)\subset \R^m\times [0, \infty)^n$) for all $(x, y) \in \R^m\times \R^n$, then $\partial_2 g(x, y)\subset (-\infty, 0]^n$ ($\partial_2 g(x, y)\subset [0, \infty)^n$) for every $(x, y)\in \R^m\times \R^n$.  
\item[$iii)$] If there exists $\hat{x}\in \R^m$ such that $\partial_2 g(\hat{x}, y)\subset (-\infty, 0]^n$ $(\partial_2 g(\hat{x}, y)\subset [0, \infty)^n)$ for all $y\in\R^n$, then $\partial_2 g(x, y)\subset (-\infty, 0]^n$ ($\partial_2 g(x, y)\subset [0, \infty)^n$) for all $(x, y)\in \R^m\times \R^n$.
\item[$iv)$] If there exists $(\hat{x}, \hat{y})\in \R^{m}\times \R^n$ and a constant $C>0$ such that 
\begin{align}\label{ass:bound}
g(\hat{x}, y)\leq C \qquad \text{for all $y\geq \hat{y}$ ($y\leq \hat y$),}
\end{align}
then $\partial_2 g(x, y)\subset (-\infty, 0]^n$ ($\partial_2 g(x, y)\subset [0, \infty)^n$) for all $(x, y)\in \R^m\times \R^n$.
\end{itemize}
\end{lemma}

\begin{proof}
  We will only prove the primary statements, since those in brackets follow analogously.   
Part $i)$ results directly from the single-variable case mentioned above.

Regarding $ii)$, we let $S$ be the set resulting from Lemma~\ref{lem:partial_subdifferential} and fix $(x, y)\in \R^{m}\times \R^n$. If $(x, y)\in S$, the inclusion $\partial_2 g(x, y)\subset (-\infty, 0]^n$ is an immediate consequence of Lemma~\ref{lem:partial_subdifferential}. If $(x,y)\notin S$, the density of $S_x$ in $\R^n$ yields
for any $i\in \{1, \ldots, n\}$ a sequence $(y^{(i, j)})_j\subset S_x$ such that 
\begin{align}\label{conv55}
y^{(i, j)}\to y+e_i\quad \text{as $j\to \infty$.}
\end{align}  
Recall that $e_i$ denotes the $i$th standard unit vector in $\R^n$.
Together with the monotonicity of the subdifferential operator (see~\eqref{monotone_operator}) it follows then that for every $b^{(i, j)}\in \partial_2 g(x, y^{(i, j)})\subset (-\infty, 0]^n$ and $b\in \partial_2 g(x, y)$, 
\begin{align*}
b^{(i,j)}\cdot \delta^{(i,j)} \geq b^{(i,j)}\cdot e_i + b^{(i,j)} \cdot\delta^{(i,j)}=  b^{(i, j)} \cdot (y^{(i,j)}-y) \geq b\cdot (y^{(i,j)}-y),
\end{align*}
where $\delta^{(i,j)} = y^{(i, j)}-y-e_i$.
Letting $j\to \infty$, the convergence statement~\eqref{conv55} along with the uniform boundedness 
of $b^{(i,j)}$ with regard to $j$, which is a consequence of the continuity of the subdifferential (cf.~\eqref{continuity}), implies that $b_i = b\cdot e_i\leq 0$ for every $i\in \{1, \ldots, n\}$, meaning $b\leq 0$. Hence, all subgradients in $\partial_2 g(x, y)$ are not greater than the zero vector, which is $ii)$.

For $iii)$, we observe that $g(\hat x, \cdot):\R^n\to \R$ is non-increasing according to $i)$.  
If $x\in \R^m$, $y\in \R^n$ and $(\tilde{b}, b)\in \partial g(x, y)$, then it holds for all $\tilde{y}\in \R^n$ with $\tilde{y}\geq 0$ that
\begin{align*}
g(\hat x, 0) \geq g(\hat x, \tilde y) \geq g(x, y) +  \tilde b \cdot (\hat{x} - x) + b \cdot (\tilde y - y).
\end{align*}

Since the left-hand side is independent of $\tilde y$, one has that $b\leq 0$. Hence, $\partial g(x, y)\subset \R^m\times (-\infty, 0]^n$, and therefore $\partial_2 g(x, y)\subset (-\infty, 0]^n$ for all $(x, y)\in \R^m\times \R^n$ by~$ii)$. 

To see $iv)$, let $y\in\R^n$ and consider the subdifferential inequality with respect to the second variable. If $\hat b \in \partial_2 g(\hat x, y)$, then
\begin{align*}
g(\hat{x}, \tilde y) \geq g(\hat x, y) + \hat b\cdot (\tilde y- y) \quad \text{for any } \tilde{y}\in \R^n.
\end{align*}
By letting $\tilde y\to +\infty$ componentwise, we deduce in view of the upper bound in~\eqref{ass:bound} that $\hat b\leq 0$. This proves that $\partial_2 g(\hat{x}, y) \subset (-\infty, 0]^n$ for all $y\in \R^n$, and $iv)$ results then directly from $iii)$.
\end{proof}

After suitable identifications, the above results also apply to functions defined on Cartesian products between $\R^{d\times d}$, $\Scal^{d \times d}$, and $\R^n$. Moreover, we have the following lemma, which we will use in the proof of Theorem~\ref{theo:characterization3d}.

\begin{lemma}\label{lem:aux2}
Let $g:\Scal^{d\times d}\times \Scal^{d\times d}\to \R$ be convex.

\begin{itemize} 
\item[$i)$] If $\partial g(\eps, \eta)\subset \Scal^{d\times d}\times \Scal^{d\times d}_-$ for all $\eps, \eta \in \Scal^{d\times d}$, then $\partial_2 g(\eps, \eta)\subset \Scal^{d\times d}_-$ for every $\eps, \eta\in \Scal^{d\times d}$.  
\item[$ii)$] If there exists $\hat{\eps}\in \Scal^{d\times d}$ and a constant $C>0$ such that 
\begin{align*}
g(\hat{\eps}, x\otimes x)\leq C \qquad \text{for all $x\in \R^d$,}
\end{align*}
then $\partial_2 g(\eps, \eta)\subset \Scal^{d\times d}_-$ for all $\eps, \eta\in \Scal^{d\times d}$.
\end{itemize}
\end{lemma}

\begin{proof} The proof is similar to that of the previous lemma.
  Let $\eps, \eta\in \Scal^{d\times d}$ and $B\in \partial_2g(\eps, \eta)$, as well as $x\in \R^d$ be fixed. To prove the statement $i)$, it suffices to verify that 
\begin{align}\label{B}
B:x\otimes x\leq 0.
\end{align} 
After identifying $\Scal^{d\times d}$ with $\R^{\frac{1}{2}d(d+1)}$, Lemma~\ref{lem:partial_subdifferential} implies the existence of a dense set $S_\eps\subset \Scal^{d\times d}$ whose complement is a null-set such that 
$\partial_2 g(\eps, \tilde \eta)\subset \Scal^{d\times d}_-$ for any $\tilde \eta\in S_\eps$.
Hence, we can find a sequence $(\eta^{(x, j)})_j \subset S_\eps$ 
such that 
\begin{align}\label{con_eta}
\eta^{(x, j)}\to \eta + x\otimes x\quad \text{as $j\to \infty$.}
\end{align} 
Let $B^{(x, j)} \in \partial_2 g(\eps, \eta^{(x, j)})\subset \Scal^{d\times d}_-$. It follows then along with the monotonicity of the subdifferential that 
\begin{align}\label{B2}
B^{(x,j)}:\delta^{(x, j)}\geq  B^{(x, j)}:(x\otimes x) + B^{(x,j)}:\delta^{(x, j)} = B^{(x,j)} : (\eta^{(x,j)} - \eta)\geq B:(\eta^{(x, j)} - \eta), 
\end{align}
where $\delta^{(x, j)} = \eta^{(x,j)} - \eta - x\otimes x$. Finally, we pass to the limit $j\to \infty$ in~\eqref{B2}. As $\delta^{(x, j)}\to 0$ for $j\to \infty$ by~\eqref{con_eta}, and $B^{(x, j)}$ is uniformly bounded 
with respect to $j$ due to the continuity of the subdifferential in the sense of~\eqref{continuity}, one obtains~\eqref{B} as desired.

Next, we turn to $ii)$. Let $\eta\in \Scal^{d\times d}$ and observe that for any $\hat{B}\in \partial_2 g(\hat \eps, \eta)$, $x\in\R^d$ and $t\in \R$, 
\begin{align*}
C\geq g(\hat \eps, tx\otimes tx) \geq g(\hat \eps, \eta) + t^2\hat B : (x\otimes x)  - \hat B: \eta. 
\end{align*}
Letting $t\to \infty$ implies $\hat{B} : (x\otimes x) = \hat B x\cdot x\leq 0$. This shows that $\hat{B}$ is negative semi-definite  and 
\begin{align*}
g(\hat \eps, \eta- x\otimes x)\geq g(\hat \eps, \eta) - \hat B :(x\otimes x) \geq g(\hat \eps, \eta)  \quad \text{for every $\eta\in \Scal^{d\times d}$ and $x\in \R^d$.}
\end{align*} 
Then 
for $\eps, \eta\in \Scal^{d\times d}$ and $(\tilde{B}, B) \in \partial g(\eps, \eta)$ and $x\in \R^d$,
\begin{align*}
g(\hat \eps, \eta) \geq g(\hat \eps, \eta+x\otimes x) \geq g(\eps, \eta) + \tilde{B}:(\hat{\eps}-\eps) + B:(x\otimes x).
\end{align*}
Since the left-hand side is independent of $x$, the symmetric matrix $B$ has to be negative semi-definite.
  Thus, $\partial g(\eps, \eta)\subset \Scal^{d\times d}\times \Scal^{d\times d}_-$ for all $\eps, \eta\in \Scal^{d\times d}$, and $ii)$ is an immediate consequence of $i)$. 
  \end{proof}
%%%%%%%%%%%%%%%%%%%%%%%%%%%%%%%%%%%%%%%%%%%%%%%%%%%%%%%%%%%%%%%%%%%%

\section{Symmetric polyconvexity in $2$d}\label{sec:2d}

In this section, we provide a characterization of symmetric polyconvex functions and study symmetric polyaffine functions as well as symmetric polyconvex quadratic forms in 2d. 

\subsection{Characterization of symmetric polyconvexity in 2d}
  The next theorem gives a necessary and sufficient condition for a real function on $\Scal^{2\times 2}$ to be symmetric polyconvex. While a classical polyconvex function can be expressed as a convex function of minors, symmetric polyconvex functions in the 2d case can be represented as convex functions of minors satisfying an additional monotonicity condition in the variable of the determinant.

\begin{theorem} 
\label{prop:charpoly2d}
A function $f:\Scal^{2\times 2}\to\R$ is symmetric polyconvex if and only if there exists a convex function $g:\Scal^{2\times 2}\times \R\to \R$ that is non-increasing with respect to the second variable such that
\begin{align*}%\label{repres}
f(\eps) &= g(\eps, \det \eps) \quad \text{for all $\eps\in \Scal^{2\times 2}$.}
\end{align*}
\end{theorem}

\begin{proof}
For the proof of sufficiency, let $f(\eps)=g(\eps, \det \eps)$ with $g$ as in the statement.  
Then, along with~\eqref{det_2d}, we obtain for $F\in \R^{2\times 2}$ that
\begin{align*}
\tilde{f}(F) = f(F^s)=g(F^s, \det F^s) = g(F^s, - \tfrac{1}{4}(F_{12}-F_{21})^2 + \det F) = \tilde{g}(F, \det F),
\end{align*}
where $\tilde{g}: \R^{2\times 2} \times \R\to \R$ is given as
\begin{align*}
\tilde{g}(F, t) = g(F^s, -\tfrac{1}{4}(F_{12}-F_{21})^2 + t).
\end{align*} 

By definition, the function $\tilde{g}$ is the composition of $g$ with a map that is linear in the first and concave in the second component. The fact that $g$ is convex and non-increasing in its second argument shows that $\tilde{g}$ is convex. Hence, $\tilde{f}$ is polyconvex, which yields the symmetric polyconvexity of $f$.  
 
To prove the reverse implication let $f:\Scal^{2\times 2}\to\R$ be symmetric polyconvex. Then, $\tilde{f}$ is polyconvex and  
there exists 
$\tilde{g}:\R^{2\times 2}\times\R\to\R$ convex such that 
\begin{align*}
\tilde{f}(F) =\tilde{g}(F,\det F), \qquad F\in \R^{2\times 2}.
\end{align*}
Defining $g:\Scal^{2\times 2}\times \R\to \R$ as the restriction of $\tilde{g}$ to $\Scal^{2\times 2}$ in the first variable, i.e.~$g=\tilde{g}|_{\Scal^{2\times 2}\times \R}$, $g$ is convex and $f(\eps)=g(\eps, \det \eps)$ for all $\eps\in \Scal^{2\times 2}$. It remains to show that $g$ is non-increasing in the second argument.

For $t\geq 0$, let 
\begin{align*}
F_t=\sqrt{t} (e_2\otimes e_1 -e_1\otimes e_2) = \begin{pmatrix} 0& -\sqrt{t} \\ \sqrt{t} & 0 \end{pmatrix}, 
\end{align*}
then $F_t^s=0$ and $\det F_t=t$, and we conclude from the convexity of $\tilde{g}$ that 
\begin{align}\label{eq45}
\begin{split}
 g(0,t)   & = \tilde g(F_t^s,\det F_t) = \tilde{g}(\tfrac{1}{2}F_t+\tfrac{1}{2}F_t^T, \tfrac12\det F_t + \tfrac12 \det F_t^T)  \\ & \le \frac{1}{2} \tilde{g}(F_t,\det F_t)+ \frac{1}{2} \tilde{g}(F_t^T,\det F_t^T) = \frac{1}{2} \tilde{f}(F_t) + \frac{1}{2}\tilde{f}(F_t^T) = f(F_t^s) = f(0).
\end{split}
\end{align}
For the equality before the last, we have used that $\tilde{f}(F) = f(F^s) = \tilde{f}(F^T)$ for all $F\in \R^{2\times 2}$. 
Hence,    the uniform bound in~\eqref{eq45} for $t\geq 0$ together with Lemma~\ref{lem:aux}\; $iii)$ yields that $g$ is non-increasing in the second variable. 
\end{proof}
\begin{remark}\label{rem:nonuniqueness}
a) Due to Lemma~\ref{lem:aux}\;$i)$, an equivalent way of phrasing the necessary and sufficient condition is that the convex function $g$ satisfies $\partial_2 g(\eps, t)\subset (-\infty, 0]$ for all $\eps\in \Scal^{2\times 2}$ and $t\in \R$.   

b) We point out that the representation of a function $f:\Scal^{2\times 2}\to \R$ in terms of a convex function of $\eps$ and $\det\eps$ is in general not unique. For instance, let 
\begin{align*}
g(\eps, t) = ({\rm tr\,}\eps)^2 - t\quad \text{ and }\quad h(\eps, t) = |\eps|^2 + t
\end{align*}
for $\eps\in \Scal^{2\times 2}$ and $t\in \R$, and consider
\begin{align}\label{rep33}
f(\eps)=g(\eps, \det \eps) = ({\rm tr\,} \eps)^2-\det\eps, \quad \eps\in \Scal^{2\times 2}.
\end{align} 
Because of the identity
\begin{align*}
({\rm tr\,} \eps)^2-\det\eps = \eps_{11}^2 + \eps_{22}^2 +\eps_{11}\eps_{22} +\eps_{12}^2= \abs{\eps}^2+\det\eps, 
\end{align*}
it follows that $g(\eps, \det \eps)=h(\eps, \det \eps)$ for all $\eps\in \Scal^{2\times 2}$. Hence also,
 \begin{align}\label{rep44}
f(\eps)=h(\eps, \det \eps) = \abs{\eps}^2+\det\eps, \quad \eps\in \Scal^{2\times 2},
\end{align}  
even though $h\neq g$;
in particular, $g$ is decreasing in $t$, while $h$ is increasing in $t$. 
In view of~\eqref{rep33}, Theorem~\ref{prop:charpoly2d} clearly shows that function $f$ is symmetric polyconvex, whereas Theorem~\ref{prop:charpoly2d} does not allow for an immediate conclusion when looking at~\eqref{rep44}. 

Generally speaking, it depends on the way a function $f:\Scal^{2\times 2}\to \R$ is given, whether deciding about symmetric polyconvexity of $f$ is immediate or not directly obvious. 
If, however, $f$ can be expressed as a convex function depending on $\det \eps$ only, then Corollary~\ref{cor:char2d} below gives a simple criterion. % this representation in~\eqref{repres} is unique, cf.\ Lemma~\ref{lem:identical2}. \color{black}
\end{remark}

%\color{magenta}
%\begin{remark}\label{rem:nonuniqueness}
%\color{magenta}We emphasize that representations of $f$ as functions of $\eps$ and $\det\eps$ are not unique. For instance, let 
%\begin{align*}
%g(\eps, t) = ({\rm tr\,}\eps)^2 - t\quad \text{ and }\quad h(\eps, t) = |\eps|^2 + t
%\end{align*}
%for $\eps\in \Scal^{2\times 2}$ and $t\in \R$.
%Because of the identity
%\begin{align*}
%({\rm tr\,} \eps)^2-\det\eps = \eps_{11}^2 + \eps_{22}^2 +\eps_{11}\eps_{22} +\eps_{12}^2= \abs{\eps}^2+\det\eps, 
%\end{align*}
% it follows that $g(\eps, \det \eps)=h(\eps, \det \eps)$ for all $\eps\in \Scal^{2\times 2}$. Yet, \color{magenta}$h(\eps,t)\neq g(\eps,t)$, \color{black} in particular, $g$ is decreasing in $t$, while $h$ is increasing in $t$. 
%
% By Theorem~\ref{prop:charpoly2d},   the function
%\begin{align*}
%f(\eps)=({\rm tr\,} \eps)^2-\det\eps=   \abs{\eps}^2+\det\eps, \quad \eps\in \Scal^{2\times 2},
%\end{align*}
%is symmetric polyconvex. Depending on which representation is chosen, this is either immediate or not directly obvious. \color{magenta} If, however, $f$ can be represented by a function depending on $\det \eps$ only, then this representation is unique, cf.\ Lemma~\ref{lem:identical2}. \color{black}
%\end{remark}
%\color{black}
%We point out that replacing the right-handside $h(\det \eps)$ in~\eqref{eq92} by $h(\eps, \det \eps)$, where now $h:\Scal^{2\times 2}\times \R\to \R$, does not imply $h=g$. Here is a simple counterexample. 

In Example~\ref{ex:pmdet}, we convinced ourselves that $\eps\mapsto - \det \eps$ is symmetric polyconvex, whereas $\eps\mapsto \det \eps$ is not. As a consequence of Theorem~\ref{prop:charpoly2d}, the following more general result can be obtained.

\begin{corollary}\label{cor:char2d}
Let $h:\R\to \R$ and let $f:\Scal^{2\times 2}\to \R$ be given by 
\begin{align*}
f(\eps) = h(\det \eps), \quad \eps\in \Scal^{2\times 2}.
\end{align*}
Then $f$ is symmetric polyconvex if and only if $h$ is convex and non-increasing. 
\end{corollary}

Indeed, in light of the next lemma the proof follows immediately.
\begin{lemma}\label{lem:identical2}
Let $g:\Scal^{2\times 2}\times \R\to \R$ be convex and $h:\R\to \R$.
If
\begin{align*}%\label{eq92}
g(\eps, \det \eps) = h(\det \eps)\quad \text{for all $\eps\in \Scal^{2\times 2}$,}
\end{align*}
then $g$ is constant in its first variable and $h(t)=g(0, t)$ for  every $t\in\R$. 
\end{lemma}

\begin{proof}
In order to show that $g=g(\eps, t)$ is constant in $\eps$ for any $t\in\R$ , we use~\eqref{convexhull_2d} to express $\eps\in\Scal^{2\times 2}$ as the convex combination of two symmetric matrices $\eps_1, \eps_2$ with vanishing determinant, i.e.~$\eps=\lambda \eps_1 + (1-\lambda) \eps_2$ with $\lambda\in [0,1]$. Then, by the convexity of $g$,
\begin{align*}
g(\eps, 0) \leq \lambda g(\eps_1, 0) + (1-\lambda)g(\eps_2, 0) = h(0) \quad\text{ for any } \eps \in \Scal^{2\times 2}.
\end{align*}
 The independence of $g$ of the first argument follows from Lemma~\ref{lem:aux}\;$iv)$. 

To prove the second part of the statement, observe that for every $t\in \R$ one can find a symmetric matrix whose determinant equals $t$. 
\end{proof}

At the end of the section, we turn in more detail to two classes of polyconvex functions.

\subsection{Symmetric polyconvex quadratic forms.} \label{sec:spc2}
In two dimensions, the three classes of symmetric polyconvex, quasiconvex, and rank-one convex quadratic forms are identical. In view of Definition~\ref{def:sqc}, this is a consequence of the corresponding well-known property of classical semi-convex quadratic forms~\cite[Theorem~5.25]{Dac08}. 
The following characterization constitutes a refinement for the symmetric case of a result by Marcellini~\cite[Eqn.\,(11)]{Mar84}, see also~\cite[Lemma~5.27]{Dac08}.  
 
\begin{proposition}[Characterization of symmetric polyconvex quadratic forms in 2d]\label{prop:quadforms2d}
Let $f:\Scal^{2\times 2}\to\R$ be a quadratic form. Then the three following statements are equivalent:
\begin{itemize}
\item[$i)$] f is symmetric polyconvex;
\item[$ii)$] there exist $\alpha\geq 0$ and a convex quadratic form $h:\Scal^{2\times 2}\to \R$ such that 
$$
f(\eps)=h(\eps)-\alpha\det\eps \quad \text{for all $\eps\in \Scal^{2\times 2}$;}
$$
\item[$iii)$] there exists $\alpha\geq 0$ such that
\begin{align*}
f(\eps) +\alpha \det\eps\geq 0 \quad \text{for all $\eps\in \Scal^{2\times 2}$.}
\end{align*}
\end{itemize}
\end{proposition}

\begin{proof}
Since a quadratic form is convex if and only if it is non-negative, $ii)$ and $iii)$ are equivalent.
The implication ``$ii)\Rightarrow i)$'' is an immediate consequence of Theorem~\ref{prop:charpoly2d}.
To see that $i)$ implies $iii)$, we use that
by \cite[Lemma~5.27, p.192]{Dac08} there exists $\alpha \in\R$ such that $\tilde f(F) + \alpha \det F\geq 0$, $F\in \R^{2\times 2}$. Hence, $\tilde{h}:\R^{2\times 2}\to \R$ defined by $\tilde h(F) = \tilde f(F) + \alpha \det F$ is quadratic and thus convex. Note that in particular, $f(\eps) + \alpha \det \eps\geq 0$ for all $\eps\in \Scal^{2\times 2}$.

Next we will show that in the symmetric setting we indeed obtain that $\alpha\geq 0$.  The argument is based on the observation that the quadratic form $q_\alpha:\R^{2\times 2}\to \R$ defined for $\alpha\in \R$ by
\begin{align*}
q_\alpha(F) =\alpha \det F^a = \frac{\alpha}{4}(F_{12} -F_{21})^2
\end{align*} 
is convex if and only if $\alpha\geq 0$. 
To show that $q_\alpha$ is convex, 
we infer in view of~\eqref{det_2d} that 
\begin{align*}
q_\alpha(F)=  \alpha (\det F-\det F^s) = \tilde{h}(F) - \tilde{f}(F) +  f(F^s) - h(F^s)= \tilde{h}(F) -  h(F^s)
\end{align*}
for all $F\in \R^{2\times 2}$. Exploiting that $q_\alpha(F)=q_\alpha(F^a)$ then leads to
\begin{align*}
q_\alpha(F) = \tilde{h}(F^a), \quad F\in \R^{2\times 2}.
\end{align*}
Since the composition of a convex with a linear function is always convex, $q_\alpha$ is convex, and hence, $\alpha\geq 0$ as asserted in $iii)$.
\end{proof}

\subsection{Symmetric polyaffine functions}\label{subsec:polyaffine}
 If $f:\Scal^{2\times 2}\to \R$ is symmetric polyaffine, that is both $f$ and $-f$ are symmetric polyconvex, then according to Theorem~\ref{prop:charpoly2d}, there exist two convex functions $g_+, g_-$ that are non-increasing in their second argument such that
\begin{align*}
f(\eps) = g_+(\eps, \det\eps) = -g_-(\eps, \det \eps), \quad \eps\in \Scal^{2\times 2}.
\end{align*}
We apply Lemma~\ref{lem:identical2} with $g=g_++g_-$ and $h$ the zero function to find that $g_-=-g_+$. Hence, $g_-$ is both non-increasing and non-decreasing, and therefore constant, with respect to the last variable. Due to the convexity of $g_+(\eps, 0)$ and $g_-(\eps, 0)$, $f$ has to be affine.

Summing up, we observe that $f$ is symmetric polyaffine if and only if it is affine. 

%%
%%%%%%%%%%%%%%%%%%%%%%%%%%%%%%%%%%%%%%%%%%%%%%%%%%%%%%%%%%%
%%

\section{Symmetric polyconvexity in $3$d}\label{sec:3d}

This section is devoted to a discussion of symmetric polyconvex functions in 3d. After providing a characterization of symmetric polyconvexity, we subsequently discuss a few examples and two important subclasses of symmetric polyconvex functions in three dimensions, these are symmetric polyconvex quadratic forms and symmetric polyaffine functions.

\subsection{Characterization of symmetric polyconvexity in $3$d}

The following theorem constitutes the main result of this paper.

\begin{theorem}[Characterization of symmetric polyconvexity in 3d]\label{theo:characterization3d}
A function  $f:\Scal^{3\times 3}\to\R$ is symmetric polyconvex if and only if  there exists a convex function $g:\Scal^{3\times 3}\times \Scal^{3\times 3} \to \R$
 such that $\partial_2 g (\eps, \eta)\subset \Scal^{3\times 3}_-$ for every $\eps, \eta \in \Scal^{3\times 3}$   and
 \begin{equation}\label{representation}
f(\eps) = g(\eps, \cof \eps) \quad \text{for all } \eps\in \Scal^{3\times 3}.
\end{equation}
\end{theorem}

\begin{proof} We subdivide the proof into the natural two steps, proving first that~\eqref{representation} is necessary for symmetric polyconvexity, and then that it is also sufficient.

{\em Step~1: Necessity.}
Since $f$ is symmetric polyconvex, $\tilde f$ defined as in Definition~\ref{def:sqc} is polyconvex. Hence, there is a convex function $\tilde g: \R^{3\times 3}\times \R^{3\times 3}\times \R\to \R$ such that
\begin{align*}
\tilde{f}(F) = \tilde{g}(F, \cof F, \det F), \quad F\in \R^{3\times 3}.
\end{align*}
Restricting $\tilde g$ in the first two variables to $\Scal^{3\times 3}$ 
gives a convex function $g$ 
defined on $\Scal^{3\times 3}\times \Scal^{3\times 3} \times \R$ that satisfies
\begin{align*}
f(\eps) = g(\eps, \cof \eps, \det \eps)\quad \text{for all $\eps\in \Scal^{3\times 3}$,}
\end{align*}
in view of~\eqref{cof_3d}.

Next, we show that $g$ is constant in the determinant. Then, in Step~1b, the asserted negative semi-definiteness of matrices in the subdifferential with respect to the second variable of $g$ is established.

{\em Step~1a: $g$ is constant in the third variable.} 
We will apply Lemma~\ref{lem:aux}\;$iv)$ to $g$ with $\hat x = (0,0) \in \Scal^{3\times 3} \times \Scal^{3\times 3}$ and $\hat{y} = 0\in \R$.   To prove the uniform upper bound needed there, the estimate 
\begin{align}\label{eq}
g(\eps_\delta, te_3\otimes e_3, \delta t) \leq f(\eps_\delta) \quad \text{for all $t\geq 0$ and $\delta\in \R$, }
\end{align} 
where $\eps_\delta=\delta e_3\otimes e_3$, will turn out useful. Indeed, with
$$F_{t, \delta}=\eps_ \delta+ \sqrt{t}(e_1\otimes e_2-e_2\otimes e_1)
= \begin{pmatrix} 0 & \sqrt{t} & 0\\
- \sqrt{t} & 0 & 0 \\
0 & 0 & \delta
\end{pmatrix}, $$
one has that $F_{t, \delta}^s=\eps_\delta$, $(\cof F_{t, \delta})^s = \frac{1}{2}(\cof F_{t, \delta}+ \cof F_{t, \delta}^T)= t e_3\otimes e_3$, and $\det F_{t, \delta}=\delta t$. 
Exploiting the symmetry and the convexity of $\tilde{g}$ then leads to
\begin{align*}
& g(\eps_\delta, t e_3\otimes e_3, \delta t) = \tilde{g}(F_{t, \delta}^s, (\cof F_{t, \delta})^s, \det F_{t, \delta}) \\ 
&\qquad \leq 
\tfrac{1}{2}\tilde{g}(F_{t, \delta}, \cof F_{t, \delta}, \det F_{t, \delta}) + \tfrac{1}{2} \tilde{g}(F_{t, \delta}^T, \cof F_{t, \delta}^T, \det F_{t, \delta}^T) \\ & \qquad = \tfrac{1}{2} \tilde{f}(F_{t, \delta}) + \tfrac12 \tilde{f}(F_{t, \delta}^T) = \tfrac12 f(F_{t, \delta}^s) + \tfrac12 f(F_{t, \delta}^s) = 
f(F_{t, \delta}^s) = f(\eps_\delta),
\end{align*}
which shows~\eqref{eq}.

Now we use the subdifferentiability of the convex function $g$.
For $r\in \R$, let $(\tilde{B}_r, B_r, b_r) \in \partial g(0, 0, r) \subset \Scal^{3\times 3} \times \Scal^{3\times 3} \times \R$. Then in conjunction with~\eqref{eq} one has that
\begin{align}\label{eqq}
& f(\eps_\delta) \geq g(\eps_\delta, te_3 \otimes e_3, \delta t) \geq g(0, 0, r)  +\tilde{B}_r : \eps_\delta+ B_r : (te_3\otimes e_3) + b_r (\delta t-r) 
\end{align} 
for all $t\geq 0$ and $\delta\in \R$. 
By letting $t$ tend to infinity, we conclude that $B_r : (e_3 \otimes e_3) + b_r\delta\leq 0$ for all $\delta$. Varying $\delta$ in $\R$ implies that $b_r=0$. 
Finally, we set $t=0$ and $\delta=0$ in~\eqref{eqq} to obtain
\begin{align*}
g(0, 0, r) \leq f(0) \quad \text{for all $r\in \R$.}
\end{align*}
 According to Lemma~\ref{lem:aux}\;$i)$ and $iv)$, $g$ is then non-increasing and non-decreasing,   and thus constant, in the last variable. 
For simplicity of notation, this third variable is from now on omitted in the functions $\tilde{g}$ and $g$.

 {\em Step~1b: Partial subdifferential condition for $g$.}
We will show that
\begin{align}\label{tt2}
g(0, x\otimes x)\le g(0, 0) \quad \mbox{ for any } x\in \R^3.
\end{align}
Then  all partial subgradients with respect to the second variable of $g$ are negative semi-definite as a consequence of Lemma~\ref{lem:aux2}\,$ii)$.

To prove~\eqref{tt2}, let $F_{x}=X$ with the anti-symmetric matrix
\begin{align*}
X=\begin{pmatrix} 0 & x_3 & -x_2\\ -x_3 & 0 & x_1 \\ x_2& -x_1 & 0 \end{pmatrix}, \color{black} 
\end{align*}
%
%$$
%X=\begin{pmatrix}
%0&x_1&-x_2\\
%-x_1&0&x_3\\
%x_2&-x_3&0
%\end{pmatrix}.
%$$  
Then, $F_{x}^s=0$ and 
$(\cof F_{x})^s = \frac12(\cof F_{x}+\cof F_{x}^T) = x\otimes x = \cof X$ by~\eqref{cofFa} and~\eqref{cof_3d}. 
Since the function $\tilde g$ is convex, we infer that 
\begin{align*}
g(0, x\otimes x) & =  \tilde g(F_{x}^s,(\cof F_{x})^s) 
\le\tfrac12 \tilde g(F_{x},\cof F_{x})+\tfrac12 \tilde g(F_{x}^T,\cof F_{x}^T)\\ & = \tfrac 12 \tilde f(F_{x}) + \tfrac12 \tilde f(F_{x}^T) = f(F_{x}^s)=g(0,0), \nonumber
\end{align*}
which is~\eqref{tt2}.

{\em Step 2: Sufficiency.} Suppose that $f$ is of the form \eqref{representation} with $g$ as in the statement. 
To prove that $\tilde f$ is polyconvex, it suffices, according to~\cite[Theorem~5.6]{Dac08}, to show the following two conditions:
\begin{itemize}
\item[$(i)$] There exists a convex function $k:\R^{3\times 3}\times \R^{3\times 3}\times \R\to \R$ such that $\tilde{f}(F)\geq k(F, \cof F, \det F)$ for all $F\in \R^{3\times 3}$.
\item[$(ii)$] Let $F,F_i \in \R^{3\times 3}$ and $\lambda_i\in [0,1]$ with $i=1,2,\dots, n=20$ such that 
\begin{align}
\label{eqpol}
\sum_{i=1}^{n} \lambda_i=1,\ \  F=\sum_{i=1}^{n}\lambda_iF_i,\ \  \cof F=\sum_{i=1}^{n}\lambda_i\cof F_i, \quad \text{and}  \ \  \det F=\sum_{i=1}^{n}\lambda_i\det F_i.
\end{align}
Then, $\tilde f(F)\le \sum_{i=1}^n \lambda_i \tilde f(F_i)$.
\end{itemize}

Let $F\in \R^{3\times 3}$. 
  With $(\tilde{B}, B) \in \partial g(0, 0) \subset \partial_1 g(0,0) \times \partial_2 g(0,0)\subset \Scal^{3\times 3} \times \Scal^{3\times 3}_-$,
condition $(i)$ follows from 
\begin{align*}
\tilde{f}(F) & = g(F^s, \cof F^s) \geq g(0, 0)  + \tilde B: F^s+ B: (\cof F)^s - B:\cof F^a \\ &= g(0, 0) + \tilde B: F^s + B : (\cof F)^s + q_{-B}(F) = k(F, \cof F, \det F),
\end{align*}
where $q_{-B}$ is as in \eqref{QA} and \color{black} $k(F, G ,t) = g(0, 0) + \tilde B: F^s + B:G^s + q_{-B}(F)$ for $F, G\in \R^{3\times 3}$, $t\in \R$. Since $-B$ is positive semi-definite, $q_{-B}$ is convex by Lemma~\ref{convexquadra}, and so is $k$.\color{black}  

To prove $(ii)$, let $F$ satisfy \eqref{eqpol} and let $B\in \partial_2 g(F^s, \cof F^s)\subset \Scal^{3\times 3}_-$. Owing to the concavity of $q_B$ (again by Lemma~\ref{convexquadra})
we then find that
\begin{align*}
B:\Bigl(\cof F^a-\sum_{i=1}^n\lambda_i\cof F^a_i\Bigr)=q_B(F)-\sum_{i=1}^n\lambda_iq_B ({F_i})\ge 0.
\end{align*}
Along with the subgradient property of $B$, this implies 
\begin{align*}
\begin{split}
g\Bigl(F^s,\cof F^s+\cof F^a-\sum_{i=1}^n \lambda_i\cof F^a_i\Bigr) & \geq g(F^s,\cof F^s)+B:\Bigr(\cof F^a-\sum_{i=1}^n\lambda_i\cof F^a_i\Bigl) \\& \geq g(F^s, \cof F^s).
\end{split}
\end{align*}
Then, since $\cof F^s+\cof F^a=(\cof F)^s = \sum_{i=1}^n \lambda_i (\cof F_i)^s$ by \eqref{eqpol} and $g$ is convex, we infer that 
\begin{align*}
\tilde f(F) =g(F^s,\cof F^s) &\le g\Bigl(F^s,\cof F^s+\cof F^a-\sum_{i=1}^n\lambda_i\cof F_i^a\Bigr) \\ & =g\Bigl(F^s,\sum_{i=1}^n\lambda_i(\cof F_i)^s-\sum_{i=1}^n\lambda_i\cof F_i^a\Bigr)
=g\Bigl(\sum_{i=1}^n \lambda_iF_i^s,\sum_{i=1}^n\lambda_i\cof F_i^s\Bigr)\\ &
\le\sum_{i=1}^n \lambda_i g(F_i^s,\cof F_i^s)=\sum_{i=1}^n\lambda_i\tilde f(F_i),
\end{align*}

which finishes the proof.
\end{proof}

\begin{remark}\label{remrepre}
a) As a consequence of Theorem~\ref{theo:characterization3d}, any symmetric polyconvex function $f$ satisfies a monotonicity assumption with regard to the diagonal minors. Precisely, there exists a convex function $g:\Scal^{3\times 3}\times \R^3\times \R^3 \to \R$ that is non-increasing in its second argument such that
\begin{align*}
f(\eps) = g(\eps, \cofdiag \eps, \cofoff \eps), \quad \eps\in \Scal^{3\times 3},
\end{align*}
cf.~Section~\ref{subsec:notations} for notations.   In view of~Lemma~\ref{lem:aux}\,$i)$ and $ii)$,
this follows from the fact that diagonal entries of negative semi-definite matrices are always non-positive.

b) Note that a function $f:\Scal^{3\times 3}\to \R$ given as in~\eqref{representation} with a convex function $g$ whose partial subdifferential regarding the second variable is not negative semi-definite may still be symmetric polyconvex. 
Indeed, 
\begin{align*}
f(\eps) =  \eps_{11}^2+ \eps_{22}^2 + 2\eps_{12}^2 + (\cof\eps)_{33}, \quad \eps\in \Scal^{3\times 3},
\end{align*}
which depends increasingly on the diagonal cofactors,
admits an alternative representation that is in accordance with Theorem~\ref{theo:characterization3d}, namely
\begin{align}\label{falter1}
f(\eps) = (\eps_{11} +\eps_{22})^2 - (\cof \eps)_{33}, \quad \eps\in \Scal^{3\times 3}.
\end{align}
Thus, $f$ is in fact symmetric polyconvex. 

c) We emphasize that the representation of a symmetric polyconvex function according to Theorem~\ref{theo:characterization3d} is not unique. To see this, take for example $f$ as in~\eqref{falter1} and observe that it can equivalently be expressed as
\begin{align*}\label{falter2}
f(\eps) = \tfrac{1}{2}(\eps_{11}  + \eps_{22})^2 + \tfrac{1}{2} \eps_{11}^2 + \tfrac{1}{2}\eps_{22}^2 +\eps_{12}^2,\quad \eps\in \Scal^{3\times 3}.
\end{align*} 
In particular, $f$ is even convex. 
\end{remark}

The $3$d analogon of Corollary~\ref{cor:char2d} is based on an auxiliary result extending Lemma~\ref{lem:identical2}.

\begin{lemma}\label{lem:identical3}
Let $g:\Scal^{3\times 3}\times \Scal^{3\times 3}\to \R$ be convex.
\begin{itemize} 
\item[$i)$] If there exists a function $h: \R\to \R$ such that 
\begin{align*}
g(\eps, \cof \eps) = h(\det \eps)\quad \text{for all $\eps\in \Scal^{3\times 3}$,}
\end{align*}
then $h$ is constant. 
\item[$ii)$] If there exists a continuous function $h:\Scal^{3\times 3}\to \R$ satisfying 
\begin{align*}
g(\eps, \cof \eps) = h(\cof \eps)\quad \text{for all $\eps\in \Scal^{3\times 3}$,}
\end{align*}
then $h=g(0, \cdot)$.
\end{itemize} 
\end{lemma}

\begin{proof}
First, we argue that, both in $i)$ and $ii)$, the function $g$ is independent of the first variable. This follows as in the proof of Lemma~\ref{lem:identical2}, just replacing~\eqref{convexhull_2d} with~\eqref{convexhull_3d}. Notice that every matrix with vanishing cofactor has in particular zero determinant.

In the following, let $g_0 = g(0, \cdot):\Scal^{3\times 3}\to \R$. As a consequence of the separate convexity of $g$, the auxiliary function $g_0$ is convex, and we have that $g_0(\cof \eps) = g(0, \cof\eps) = g(\eps, \cof \eps)$ for any $\eps \in \Scal^{3\times 3}$.

As for $i)$, we use Lemma~\ref{lem:aux}\;$iii)$ to infer from the observation that $g_0(te_i\otimes e_i) = g_0(\cof(te_j\otimes e_j + e_k\otimes e_k)) = h(0)$ for all $t\in \R$ and any circular permuation $(ijk)$ of $(123)$ that $g_0$ is constant in the variables corresponding to the diagonal entries. Thus, plugging diagonal matrices into~the identity $g_0(\cof \eps) = h(\det \eps)$ for $\eta\in \Scal^{3\times 3}$ shows that $h$ is constant.

Regarding $ii)$, we observe that the image of the map ${\rm cof}:\Scal^{3\times 3}\to \Scal^{3\times 3}$ contains the set of symmetric invertible matrices. 
Indeed, for any $\eta\in \Scal^{3\times 3}$ with $\det \eta \neq 0$, one finds by~\eqref{Cramer} that
\begin{align*}
\cof \bigl((\det \eta)^{\frac{1}{2}}\eta^{-1}\bigr) = \eta,
\end{align*}
and consequently,
\begin{align*}
 h(\eta) =  h \bigl( \cof \bigl((\det \eta)^{\frac{1}{2}}\eta^{-1}\bigr)\bigr) 
 = g\bigl((\det \eta)^{\frac{1}{2}}\eta^{-1}, \eta\bigr)  = g_0(\eta).
\end{align*}

Finally, since the set of symmetric invertible matrices lies dense in $\Scal^{3\times 3}$, we use the continuity of $h$ and $g_0$ to conclude that $h(\eta)=g_0(\eta)$ for all $\eta\in \Scal^{3\times 3}$, which yields $ii)$.  
\end{proof}

\begin{remark}
a) The statement in Lemma~\ref{lem:identical3} $ii)$ clearly fails, if we do not require $h$ to be continuous. To see this, we first recall that, in the 3d case, the rank of any cofactor matrix is not equal to 2. Indeed, if $\rank\eps=0$ or $\rank\eps=1$, then $\cof\eps =0$, so its rank is 0. If $\rank\eps=2$, we have by Cramer's rule \eqref{Cramer} and Sylvester's rank formula that $\rank \cof \eps + \rank \eps\leq 3$ 
and hence, $\rank\cof\eps \leq1$. Finally, if $\eps$ is invertible, so is $\cof\eps$ and $\rank\cof\eps =3$.
Now,  let $g:\Scal^{3\times 3}\times \Scal^{3\times 3}\to \R$ be the zero function, and choose $h:\Scal^{3\times 3}\to \R$ such that $h(\eps) = 1$ for all $\eps \in \Scal^{3\times 3}$ with $\rank \eps = 2$, and $h(\eps)=0$ otherwise. Then, since $\rank\cof\eps \neq2$, $g(\eps, \cof\eps) = h(\cof\eps)$ for any $\eps \in \Scal^{3\times 3}$, but $h\neq 0 = g(0,\cdot)$. 

b) Let $g:\Scal^{3\times 3}\times \Scal^{3\times 3}\to \R$ be convex. If there is a continuous function $h:\Scal^{3\times 3}\times \R\to \R$ such that
\begin{align*}
g(\eps, \cof \eps) = h(\cof \eps,\det \eps)\quad \text{for all $\eps\in \Scal^{3\times 3}$},
\end{align*}
results analogous to Lemma~\ref{lem:identical3} $i)$ and $ii)$ cannot be expected. That is, $h$ may be non-constant in the last variable and $h(\cdot, 0) \neq g(0,\cdot)$.   Due to $\det (\cof \eps) = (\det \eps)^2$ for all $\eps \in \Scal^{3\times 3}$, considering the functions
\begin{align*}
g(\eps, \eta)= 0\quad\text{ and }\quad h(\eta, t) = t^2-\det \eta,
\quad \eps, \eta\in \Scal^{3\times 3}, t\in \R,
\end{align*}  
is an example. 
\end{remark}

The following statement is an immediate consequence of Theorem~\ref{theo:characterization3d} and Lemma~\ref{lem:identical3}. It allows us to decide in special situations, if a function is symmetric polyconvex or not, without having to find a function $g$ as required in Theorem~\ref{theo:characterization3d}.

\begin{corollary}\label{cor:3d}
$i)$ Let $f:\Scal^{3\times 3}\to \R$ be given by 
\begin{align*}
f(\eps) = h(\det \eps), \quad \eps\in \Scal^{3\times 3}. 
\end{align*}
with $h:\R \to \R$. 
Then $f$ is symmetric polyconvex if and only if $h$ is constant.

$ii)$ Let $f:\Scal^{3\times 3}\to \R$ be 
given by
\begin{align*}
f(\eps) = h(\cof \eps), \quad \eps\in \Scal^{3\times 3},
\end{align*}
with $h:\Scal^{3\times 3} \to \R$ continuous. Then $f$ is symmetric polyconvex,
if and only if $h$ is convex with $\partial h(\eta)\subset \Scal_-^{3\times 3}$ for all $\eta\in \Scal^{3\times 3}$.  
\end{corollary}

In contrast to the 2d setting, Corollary~\ref{cor:3d} $i)$ shows that there exists no symmetric polyconvex function in 3d that is given as a non-constant function of the determinant only.\\ 
%
%%%%%%%%%%%%%%%%%%%%%%%%%%%%%%%%%%%%%%%%%%%%%%%%%%%%%%%%%%%%%%%
%
\subsection{Symmetric polyconvex quadratic forms}  \label{sec:spc3}
  Finally, we turn our attention to quadratic forms in 3d. As a consequence of  Theorem~\ref{theo:characterization3d}, we obtain a characterization of symmetric polyconvexity for this class of functions, which reminds of the characterization of polyconvex quadratic forms in \cite[p.~192]{Dac08}.

\begin{proposition}[Characterization of symmetric polyconvex quadratic forms in 3d]\label{theoremquad}
Let $f:\Scal^{3\times 3}\to \R$ be a quadratic form. Then the following statements are equivalent: 
\begin{itemize}
\item[$i)$] $f$ is symmetric polyconvex;  
\item[$ii)$] there exist a convex quadratic form $h:\Scal^{3\times 3}\to \R$ and a matrix $A\in \Scal^{3\times 3}_+$ such that 
\begin{equation}\label{quad}
f(\eps)  = h(\eps)-A:\cof \eps \quad \text{for all $\eps\in \Scal^{3\times 3}$};
\end{equation} 
\item[$iii)$]  there is $A\in \Scal^{3\times 3}_+$ such that
\begin{align*}
f(\eps) +  A:\cof \eps \geq 0\quad \text{for all $\eps\in \Scal^{3\times 3}$.} 
\end{align*}
\end{itemize}
\end{proposition}

\begin{proof} 
Note that $ii)$ and $iii)$ are equivalent since quadratic forms are convex if and only if they are non-negative. To prove ``$ii)\Rightarrow i)$" suppose that $f$ is of the form~\eqref{quad}. Defining $g:\Scal^{3\time 3}\times \Scal^{3\time 3}\to\R$ by $g(\eps,\eta)=h(\eps)-A:\eta$ for $\eps, \eta\in \Scal^{3\times 3}$, the representation~\eqref{representation} is immediate. Moreover, $g$ is convex and smooth with 
\begin{align*}
\partial_2 g(\eps, \cof \eps) = \{-A\} \quad \text{for every $\eps\in \Scal^{3\times 3}$}.
\end{align*}
Since $-A$ is negative semi-definite by assumption, $f$ is symmetric polyconvex according to 
Theorem \ref{theo:characterization3d}.

To show that $i)$ implies $iii)$, let $f$ be a symmetric polyconvex quadratic form on $\Scal^{3\times 3}$, and let $g$ be an associated convex function resulting from Theorem \ref{theo:characterization3d}. 
As $g$ is convex, there exist $B, \tilde{B}\in \Scal^{3\times 3}$ such that $(\tilde{B}, B)\in \partial g(0,0)$. We note that owing to $\partial_2 g(0,0)\subset \Scal^{3\times 3}_-$, the matrix $B$ is negative semi-definite, cf.~\eqref{inclusion}.
Then, it follows for every $\eps\in \Scal^{3\times 3}$ that 
 \begin{align*}
 f(\eps)=g(\eps,\cof\eps) & \geq g(0,0)+\tilde{B}:\eps+B :\cof\eps,\\
 f(-\eps)=g(-\eps,\cof\eps) & \geq g(0,0)-\tilde{B} :\eps +B:\cof\eps.
 \end{align*} 
Since $f$ is a quadratic form, $f(\eps)=f(-\eps)$ for all $\eps\in \Scal^{3\times 3}$ and $g(0,0)=f(0)=0.$ By summing up the two inequalities and setting $A=-B$, we conclude that $f(\eps)+A:\cof\eps\geq 0$ for all $\eps\in \Scal^{3\times 3}$, as asserted.    
\end{proof}

Next we present an example of a quadratic form which is symmetric rank-one convex but not symmetric polyconvex. As already mentioned in the introduction, this is motivated by a corresponding result in the classical setting by Serre \cite{Ser83}.

\begin{theorem}\label{theo:counterexample}
There exists $\eta>0$ such that the quadratic form $f:\Scal^{3\times 3}\to \R$ given by
$$
f(\eps)= (\eps_{13}-\eps_{23})^2+(\eps_{12}-\eps_{13})^2+(\eps_{12}-\eps_{23})^2 + \eps_{11}^2+\eps_{22}^2+\eps_{33}^2-\eta|\eps|^2
$$
is symmetric rank-one convex, but not symmetric polyconvex.  
\end{theorem}

\begin{proof}    Consider the quadratic form 
\begin{align}\label{f0}
f_0(\eps)= (\eps_{13}-\eps_{23})^2+(\eps_{12}-\eps_{13})^2+(\eps_{12}-\eps_{23})^2+ \eps_{11}^{2}+\eps_{22}^2+\eps_{33}^2, \quad \eps\in \Scal^{3\times 3},
\end{align}
and let \begin{equation}\label{eta}
\eta=\min_{\eps \in \Mcal} f_0(\eps)
\end{equation}
with 
\begin{align}\label{calM}
\Mcal=\left\{\eps\in \Scal^{3\times 3}: |\eps|=1, \eps = a\odot b \text{ with $a, b\in \R^3$}\right\}.
\end{align}
Note that the minimum in~\eqref{eta} indeed exists, and that $\eta>0$. To see the latter, assume to the contrary that $\eta =0$. Then, if $\bar{\eps}$ is a minimizer of $f_0$ in $\Mcal$, it holds that 
$f_0(\bar\eps) =0$. Along with $|\bar{\eps}|=1$ we conclude that necessarily, 
\begin{align*}
\bar\eps= \frac{1}{\sqrt{6}} \begin{pmatrix} 0 & 1 & 1 \\ 1 & 0 & 1\\ 1 & 1 & 0\end{pmatrix}.
\end{align*}
However, since this $\bar\eps$ has full rank, $\bar\eps\notin \Mcal$, which is the desired contradiction. Furthermore, observe that $\eta < 1$ as shown in \eqref{minimizer} below.

The remainder of the proof is subdivided into the natural two steps.\\

\textit{Step~1: $f$ is symmetric rank-one convex.} 
Due to~\eqref{eta}, one has that $f(a\odot b) = f_0(a\odot b) - \eta |a\odot b|^2 \geq 0$ for all $a, b\in \R^3$. Hence, $f$ is symmetric rank-one convex by \eqref{rankform}.  \\

\textit{Step~2: $f$ is not symmetric polyconvex.} Let $A=(A_{ij})_{ij}\in \Scal^{3\times 3}_+$. In view of Proposition~\ref{theoremquad} $iii)$, it is enough to find one $\eps_A\in \Scal^{3\times 3}$ such that $f(\eps_A)+A:\cof {\eps_A}<0$. 

Let $\bar\eps=a\odot b$ with $a, b\in \R^3$ such that $|\bar\eps|=1$ be a minimizer of $f_0$ in $\Mcal$. 
In view of~$f(\bar\eps)=0$, \eqref{cof_ab}, and the positive semi-definiteness of $A$, we infer that
\begin{align*}
 f(\bar\eps) +A: \cof \bar\eps = A:\cof \bar\eps= - \frac{1}{4} A(a\times b) \cdot (a\times b) \leq 0.
\end{align*} 
Hence, if $A:\cof \bar\eps<0$, setting $\eps_A=\bar\eps$ finishes the proof. 

If $A:\cof \bar\eps=0$, we consider a perturbation of $\bar\eps$ of the form 
 \begin{align*}
 \bar\eps_\delta^{11} = \bar\eps + \delta e_1\otimes e_1,\qquad  \delta\in \R,
 \end{align*} 
 for which we have that
\begin{align*}
f(\bar\eps_{\delta}^{11})+A:\cof \bar\eps_{\delta}^{11}=(1-\eta)\delta^2+(\bar\eps_{33}A_{22}+\bar\eps_{22}A_{33}-2\bar\eps_{23}A_{23}+2(1-\eta)\bar\eps_{11})\delta.
\end{align*}
One can choose $\delta\in \R$ such that $f(\bar\eps_{\delta}^{11})+A:\cof \bar\eps_{\delta}^{11}<0$ and thus, conclude the proof, unless $\bar\eps_{33}A_{22}+\bar\eps_{22}A_{33}-2\bar\eps_{23}A_{23}+2(1-\eta)\bar\eps_{11}=0$. If the latter equation is satisfied, we perturb another entry of $\bar\eps$, say $\bar\eps_{12}$, and consider 
\begin{align*}
\bar\eps_{\delta}^{12}=\bar\eps+\delta e_1\otimes e_2, \qquad \delta\in \R.
\end{align*} 
It follows then that 
\begin{align*}
f(\bar\eps_{\delta}^{12})+A:\cof \bar\eps_{\delta}^{12} &  = \left(2(1-\eta)-A_{33}\right)\delta^2 \\ & \ \ \ +(-2\bar\eps_{12}A_{33}-2\bar\eps_{33}A_{12}+2\bar\eps_{23}A_{13}+2\bar\eps_{13}A_{23}-2\bar\eps_{13}-2\bar\eps_{23}+4\bar\eps_{12}(1-\eta))\delta.
\end{align*}
 
Hence, a choice of $\delta$ such that $f(\bar\eps_{\delta}^{12})+A:\cof \bar\eps_{\delta}^{12}<0$ is possible, except when the factor in front of $\delta$ vanishes. 
Performing this perturbation procedure for all relevant components yields a proof of Step~2, provided the entries $A_{ij}$ of $A$ do not satisfy the following six equations: 
 \begin{align*}
\begin{array}{llll}
\bar\eps_{33}A_{22}+\bar\eps_{22}A_{33}-2\bar\eps_{23}A_{23}&=-2(1-\eta)\bar\eps_{11}\\
\bar\eps_{33}A_{11}+\bar\eps_{11}A_{33}-2\bar\eps_{13}A_{13}&=-2(1-\eta)\bar\eps_{22}\\
\bar\eps_{22}A_{11}+\bar\eps_{11}A_{22}-2\bar\eps_{12}A_{12}&=-2(1-\eta)\bar\eps_{33}\\
-\bar\eps_{12}A_{33}-\bar\eps_{33}A_{12}+\bar\eps_{23}A_{13}+\bar\eps_{13}A_{23}&=\bar\eps_{13}+\bar\eps_{23}-2(1-\eta)\bar\eps_{12}\\
-\bar\eps_{13}A_{22}+\bar\eps_{23}A_{12}-\bar\eps_{22}A_{13}+\bar\eps_{12}A_{23}&=\bar\eps_{12}+\bar\eps_{23}-2(1-\eta)\bar\eps_{13}\\
-\bar\eps_{23}A_{11}+\bar\eps_{13}A_{12}+\bar\eps_{12}A_{13}-\bar\eps_{11}A_{23}&=\bar\eps_{12}+\bar\eps_{13}-2(1-\eta)\bar\eps_{23}.
\end{array}
\end{align*}

Indeed, taking into account the special structure of $\bar\eps$ as a linearized strain $a\odot b$, 
we will rule out the latter case by proving that the linear system 
\begin{align}\label{linearsystem}
Lx=c
\end{align} 
with unknown $x\in \R^6$, 
\begin{align*}
L=\begin{pmatrix}
0&a_3b_3&a_2b_2&0&0&-a_2b_3-a_3b_2\\
a_3b_3&0&a_1b_1&0&-a_1b_3-a_3b_1&0\\
a_2b_2&a_1b_1&0&-a_1b_2-a_2b_1&0&0\\
0&0&-a_1b_2-a_2b_1&-2a_3b_3&a_2b_3+a_3b_2&a_1b_3+a_3b_1\\
0&-a_1b_3-a_3b_1&0&a_2b_3+a_3b_2&-2a_2b_2&a_1b_2+a_2b_1\\
-a_2b_3-a_3b_2&0&0&a_1b_3+a_3b_1&a_1b_2+a_2b_1&-2a_1b_1
\end{pmatrix},
\end{align*} 
and 
$$
c=\begin{pmatrix}
-2(1-\eta)a_1b_1\\
-2(1-\eta)a_2b_2\\
-2(1-\eta)a_3b_3\\
a_1b_3+a_3b_1+a_2b_3+a_3b_2-2 (1-\eta)(a_1b_2+a_2b_1)\\
a_1b_2+a_2b_1+a_2b_3+a_3b_2-2 (1-\eta)(a_1b_3+a_3b_1)\\
a_1b_2+a_2b_1+a_1b_3+a_3b_1-2(1-\eta)(a_2b_3+a_3b_2)
\end{pmatrix}
$$
does not admit a solution. 
To show that~\eqref{linearsystem} is unsolvable, we split the argument in two separate cases, based on~\cite[Lemma~4.1]{Koh91}.

The first case assumes that the matrix $a\odot b$ is of rank one, which means that the vectors $a$ and $b$ are parallel. According to Lemma~\ref{lem:properties_eps0} $i)$ below, we then know that all components of $a$ and $b$  are different from $0$. We now apply the Gaussian algorithm to show that solving~\eqref{linearsystem} is equivalent to solving the reduced system
$$
\begin{pmatrix}
1&0&0&-\frac{a_1}{a_2}&-\frac{a_1}{a_3}&\frac{a_1^2}{a_2 a_3}\\
0&1&0&-\frac{a_2}{a_1}&\frac{a_2^2}{a_1 a_3}&-\frac{a_2}{a_3}\\
0&0&1&\frac{a_3^2}{a_1
a_2}&-\frac{a_3}{a_1}&-\frac{a_3}{a_2}\\
0&0&0&0&0&0\\
0&0&0&0&0&0\\
0&0&0&0&0&0
\end{pmatrix}
\begin{pmatrix}
x_1\\x_2\\x_3\\x_4\\x_5\\x_6 
\end{pmatrix}
=
\begin{pmatrix}
0\\0\\0\\0\\1\\0 
\end{pmatrix}, \qquad x\in\R^6,
$$
which is clearly not solvable.

In the second case, $a\odot b$ has rank two.   Thus $\cof (a\odot b) \neq 0$. Hence, by \eqref{cof_ab}, we can assume without loss of generality that $a_2b_1-a_1b_2\neq 0$.  If $a_3\neq0$ and $b_3\neq0$, then $a_3 (a_2 b_1-a_1 b_2) b_3\neq 0$. If $a_3=0$, then, by Lemma~\ref{lem:properties_eps0}, the coefficients $a_1,a_2$ and $b_3$ cannot vanish.  If in addition $b_1\neq0$, then $a_1(a_2b_3-a_3b_2)b_1 =a_1a_2b_1b_3\neq0$. If $b_1=0$, then $b_2\neq0$  by Lemma~\ref{lem:properties_eps0}, and we conclude that $a_2(a_1b_3-a_3b_1)b_2= a_1a_2b_2b_3\neq0$.  Hence, at least one of the expressions $a_3 (a_2 b_1-a_1 b_2) b_3$, $a_1(a_2 b_3-a_3 b_2)b_1$ and $a_2(a_1b_3-a_3b_1)b_2$ is non-zero.

Let us assume that $a_3(a_2b_1-a_1b_2)b_3\neq 0$, then the Gaussian reduced form of~\eqref{linearsystem} reads
$$
\begin{pmatrix}
1&0&0&0&\frac{-a_2 a_3 b_1^2+a_1^2 b_2 b_3}{a_3 (a_2 b_1-a_1 b_2) b_3}&-\frac{-a_1 a_3 b_1^2+a_1^2 b_1 b_3}{a_3 (a_2 b_1-a_1 b_2) b_3}\\
0&1&0&0&-\frac{a_2 a_3 b_2^2-a_2^2
b_2 b_3}{a_3 (a_2 b_1-a_1 b_2) b_3}&\frac{-a_1 a_3 b_2^2+a_2^2 b_1 b_3}{a_3 (-a_2 b_1+a_1 b_2) b_3}\\
0&0&1&0&\frac{a_3 b_2-a_2 b_3}{a_2 b_1-a_1 b_2}&\frac{a_1 b_3 -a_3b_1}{ a_2 b_1 -a_1b_2}\\
0&0&0&1&-\frac{a_2 a_3 b_1 b_2-a_1 a_2 b_2 b_3}{a_3 (a_2 b_1-a_1 b_2) b_3}&-\frac{-a_1 a_3 b_1 b_2+a_1 a_2 b_1 b_3}{a_3 (a_2 b_1-a_1 b_2) b_3}\\
0&0&0&0&0&0\\
0&0&0&0&0&0
\end{pmatrix}
\begin{pmatrix}
x_1\\x_2\\x_3\\x_4\\x_5\\x_6 
\end{pmatrix}
=
\begin{pmatrix}
0\\0\\0\\0\\1\\0 
\end{pmatrix}, \qquad x\in\R^6,$$ 
which has no solution. The other situations can be treated in the same way after a suitable permutation of lines in~\eqref{linearsystem}. Finally, this shows that~\eqref{linearsystem} does not admit a solution, and the proof is completed. 
\end{proof}

The proof of Theorem~\ref{theo:counterexample} makes use of the following auxiliary result on the structure of minimizers   of $f_0$ in $\Mcal$, cf.~\eqref{eta}.

\begin{lemma}\label{lem:properties_eps0}
Suppose that $\bar\eps$ is a minimizer of $f_0$ in $\Mcal$, where $f_0$ is defined in \eqref{f0} and $\Mcal$ in \eqref{calM}. Then $\bar\eps=a\odot b$, where $a, b\in \R^3$ each has at most one zero entry, but not in the same component. In formulas, this means that 
\begin{itemize}
\item[$i)$]  $(a_i,b_i)\neq (0,0)$ for all $i\in\{1,2,3\}$, and
\item[$ii)$]  $(a_i,a_j)\neq (0,0)$ and $(b_i,b_j)\neq (0,0)$ for all $ i, j\in\{1,2,3\}$, $i\neq j$. 
\end{itemize}
\end{lemma}

\begin{proof}
It suffices for $i)$ to show that $(a_1, b_1)\neq (0, 0)$ and for $ii)$ that $(a_1,a_2)\neq(0,0)$. The other cases follow analogously.
We start by observing that 
\begin{align}\label{minimizer}
\eta =f_0(\bar\eps) \leq f_0 \left( \frac{1}{3} \begin{pmatrix} 1\\1\\1 \end{pmatrix} \odot  \begin{pmatrix} 1\\1\\1 \end{pmatrix}\right) = \frac{1}{3}.
\end{align}

For the proof of $i)$, suppose that $a_1=b_1=0$. Then
$$
\bar\eps=\begin{pmatrix}
0&0&0\\
0&a_2b_2&\frac{1}{2}(a_2b_3+a_3b_2)\\
0&\frac{1}{2} (a_2b_3+a_3b_2) & a_3b_3
\end{pmatrix},
$$
and the fact that $|\bar\eps|=1$ allows us to conclude that 
\begin{align*}
f_0(\bar\eps)=(a_2b_2)^2+\tfrac{1}{2}(a_2b_3+a_3b_2)^2+(a_3b_3)^2= |\bar\eps|^2=1.
\end{align*} 
This is a contradiction, since the minimal value of $f_0$ is strictly smaller than $1$ by~\eqref{minimizer}. Hence $i)$ follows. 

Now assume that $a_1=a_2=0$. Then 
$$
\bar\eps=a_3\begin{pmatrix}
0&0&b_1\\
0&0&b_2\\
b_1&b_2 &b_3
\end{pmatrix}
$$
with $a_3\neq 0$ and $b\neq 0$. In view of the normalization condition, we find that $a_3^2 = \frac{1}{2b_1^2+2b_2^2+b_3^2}$. 
Hence, 
\begin{align*}
f_0(\bar\eps) = 1-2a_3^2b_1b_2 \geq 1-\frac{2b_1b_2}{2b_1^2 +2b_2^2} \geq \tfrac{1}{2},
\end{align*}
which in view of~\eqref{minimizer} is in contradiction to the minimizing property of $\bar\eps$. This yields $ii)$. 
\end{proof}

Finally, we observe that for those quadratic forms that are representable as linear combinations of $2\times 2$ minors, the notions of symmetric polyconvexity and symmetric rank-one convexity are in fact identical. 
\begin{corollary}\label{cor:B}
Let $f:\Scal^{3\times 3}\to \R$ be the quadratic form $$f(\eps)=-A:\cof \eps \quad \text{ for }\eps\in \Scal^{3\times 3}$$ with a given matrix $A\in \Scal^{3\times 3}$. Then the following three statements are equivalent: 
\begin{itemize}
\item[$i)$] $f$ is symmetric polyconvex;
\item[$ii)$] $f$ is symmetric rank-one convex;
\item[$iii)$] $A$ is positive semi-definite.
\end{itemize}
\end{corollary}

\begin{proof} By \eqref{implications}, $i)$ implies $ii)$. By Proposition~\ref{theoremquad}, we know that $iii)$ implies $i)$. 
In order to prove that $ii)$ implies $iii)$, assume that $f$ is symmetric rank-one convex. According to~\eqref{cof_3d} and~\eqref{QA},
$$
\tilde{f}(F)=f(F^s)=-A:\cof F^s= -A:\cof F + A:\cof F^a = -A:\cof F + q_A(F)
$$
for $F\in \R^{3\times 3}$. 
Since $\tilde{f}$ is rank-one convex and $F\mapsto A:\cof F$ is polyconvex, and therefore in particular rank-one convex, we infer that the quadratic form $q_A$ is also rank-one convex. Due to Lemma~\ref{convexquadra}, $A$ is then positive semi-definite, which is $iii)$. 
\end{proof}

In particular, it follows from Corollary~\ref{cor:B} that any symmetric rank-one convex quadratic form given in terms of a linear combination of off-diagonal cofactor entries can only be trivial, since there is no positive semi-definite symmetric matrix with vanishing diagonal entries. 

\subsection{Symmetric polyaffine functions} \label{subsec:polyaffine3}
In the classical setting, it is well-known that for real-valued functions defined on $\R^{3\times 3}$ the properties of being polyaffine, quasiaffine (or Null-Lagrangian) and rank-one affine are equivalent, see~e.g.~\cite[Theorem~5.20]{Dac08}. Consequently, by Definition~\ref{def:sqc} the corresponding statement is true in the symmetric setting. 

Next we show that a function $f$ in 3d is symmetric polyaffine (or symmetric quasiaffine or symmetric rank-one affine) if and only if $f$ is affine.   The proof follows from the characterization of symmetric polyconvex functions from Theorem~\ref{theo:characterization3d} in conjunction with Lemma~\ref{lem:identical3}.

\begin{proposition}[Characterization of symmetric polyaffine functions in \boldmath{$3$}d]\label{cor:polyaffine}
Let $f: \Scal^{3\times 3} \to \R$. Then $f$ is symmetric polyaffine if and only if it is affine, i.e.~there are   $B\in \Scal^{3\times 3}$ and $b\in \R^3$ such that 
\begin{align}\label{polyaff}
f(\eps) = B:\eps +b, \quad \eps\in \Scal^{3\times 3}. \
\end{align}
\end{proposition}

\begin{proof}
If the function $f$ is of the form \eqref{polyaff}, then, clearly,  it is symmetric polyaffine. 

Assume, now, that $f$ is  symmetric polyaffine. 
 
Then, according to Theorem~\ref{theo:characterization3d}, one can find two convex functions $g_+$ and $g_-$ defined on $\Scal^{3\times 3}\times \Scal^{3\times 3}$ with $\partial_2 g_+(\eps, \cof \eps)\subset\Scal^{3\times 3}_-$ and $\partial_2 g_-(\eps, \cof \eps) \subset \Scal^{3\times 3}_-$ for any $\eps\in \Scal^{3\times 3}$ such that
\begin{align}\label{eq87}
f(\eps)=g_+(\eps,\cof\eps)=-g_-(\eps,\cof\eps), \quad \text {$\eps\in\Scal^{3\times 3}$.}
\end{align}
We conclude from Lemma \ref{lem:identical3} that $g_-=-g_+$, which implies that $\partial_2 g_{\pm}(\eps,\cof\eps)=0$ for all $\eps\in \Scal^{3\times 3}$. Hence, $g_\pm$ are independent of the cofactor variable, and $f$ has to be affine in view of~\eqref{eq87}, and thus representable as in~\eqref{polyaff}.
\end{proof}

 %%%%%%%%%%%%%%%%%%%% ACKNOWLEDGEMENTS %%%%%%%%%%%%%%%%%%%%%%%%%%%%
\section*{Acknowledgements}
CK gratefully acknowledges the support by a Westerdijk Fellowship from Utrecht University. This work was initiated during research visits of CK and OB to the University of W{\"u}rzburg. The latter were funded by a DAAD travel grant. 
\color{black}

%%%%%%%%%%%%%%%%%%%%%%%%%%%% BIBLIOGRAPHY %%%%%%%%%%%%%%%%%%%%%%%%%%%%%%%%%%%

\bibliographystyle{abbrv}
\bibliography{SymmetricPolyconvexity}
\end{document}